\documentclass[10pt,reqno]{amsart}

\usepackage{amsmath}
\usepackage{amsfonts}
\usepackage{amssymb}
\usepackage{amsthm}
\usepackage{mathtools}

\usepackage{stmaryrd}

\usepackage{mathrsfs}
\usepackage{latexsym}
\usepackage{verbatim} 
\numberwithin{equation}{section}
\usepackage{enumitem}

\usepackage{graphicx}
\usepackage{subfigure}
\usepackage{epstopdf}

\usepackage{empheq}

\usepackage{ifthen} 

\provideboolean{shownotes} 
\setboolean{shownotes}{true} 
\newcommand{\margnote}[1]{
\ifthenelse{\boolean{shownotes}}%
{\marginpar{\raggedright\tiny\texttt{#1}}}%
{}%
}
\newcommand{\hole}[1]{
\ifthenelse{\boolean{shownotes}}%
{\begin{center} \fbox{ \rule {.25cm}{0cm}
\rule[-.1cm]{0cm}{.4cm} \parbox{.85\textwidth}{\begin{center}
\texttt{#1}\end{center}} \rule {.25cm}{0cm}}\end{center}}
{}
}

\theoremstyle{plain}

\newtheorem{lemma}{Lemma}[section]
\newtheorem{theorem}[lemma]{Theorem}
\newtheorem{proposition}[lemma]{Proposition}
\newtheorem{corollary}[lemma]{Corollary}

\theoremstyle{definition}

\newtheorem{remark}[lemma]{Remark}
\newtheorem{definition}[lemma]{Definition}

\theoremstyle{remark}

\usepackage{cite} 

\usepackage[colorlinks=true,urlcolor=blue,
citecolor=red,linkcolor=blue,linktocpage,pdfpagelabels,
bookmarksnumbered,bookmarksopen]{hyperref}

\usepackage{orcidlink}

\usepackage{cleveref}
\usepackage[pagewise,mathlines]{lineno}

\newcommand{\bm}{\mathbf{m}}
\newcommand{\bH}{\mathbf{H}}
\newcommand{\bh}{\mathbf{h}}

\DeclareFontFamily{U}{mathx}{}
\DeclareFontShape{U}{mathx}{m}{n}{<-> mathx10}{}
\DeclareSymbolFont{mathx}{U}{mathx}{m}{n}
\DeclareMathAccent{\widehat}{0}{mathx}{"70}
\DeclareMathAccent{\widecheck}{0}{mathx}{"71}

\newcommand{\llb}{\llbracket}
\newcommand{\rrb}{\rrbracket}

\newcommand{\E}{\mathbb{E}}
\newcommand{\R}{\mathbb{R}}
\newcommand{\C}{\mathbb{C}}
\newcommand{\Z}{\mathbb{Z}}
\newcommand{\N}{\mathbb{N}}

\newcommand{\vep}{\varepsilon}

\newcommand{\cT}{{\mathcal{T}}}
\newcommand{\cE}{{\mathcal{E}}}
\newcommand{\cL}{{\mathcal{L}}}

\newcommand{\cQ}{{\mathcal{Q}}}
\newcommand{\cR}{{\mathcal{R}}}
\newcommand{\cS}{{\mathcal{S}}}

\newcommand{\cP}{{\mathcal{P}}}

\newcommand{\cA}{{\mathcal{A}}}
\newcommand{\cJ}{{\mathcal{J}}}
\newcommand{\cB}{{\mathcal{B}}}
\newcommand{\cU}{{\mathcal{U}}}
\newcommand{\cM}{{\mathcal{M}}}

\newcommand{\ccC}{\mathscr{C}}
\newcommand{\ccB}{\mathscr{B}}
\newcommand{\ccL}{\mathscr{L}}
\newcommand{\cG}{\mathcal{G}}

\newcommand{\tcS}{\widetilde{{\mathcal{S}}}}

\renewcommand{\Re}{\mathrm{Re}\,} 
\renewcommand{\Im}{\mathrm{Im}\,}

\newcommand{\skalprod}[2]{\left \langle #1 \, , #2 \right \rangle}
\newcommand{\norm}[1]{\left \| #1 \right \|}
\newcommand{\abs}[1]{\left | #1 \right |}

\DeclareMathOperator*{\rank}{\mathrm{rank}}

\DeclareMathOperator*{\Hess}{\mathrm{Hess}\,}

\usepackage{scalerel}

\newcommand{\brt}{\bar{\theta}}

\newcommand{\varep}{\varepsilon}

\newcommand{\ptsp}{\sigma_\mathrm{\tiny{pt}}}

\newcommand{\bb}[1]{\mathbb{#1}}

\newcommand{\pld}[2]{\left \langle #1 \, , #2 \right \rangle_{L^2}}

\newcommand{\nld}[1]{\left \| #1 \right \|_{L^2}}

\DeclareMathOperator{\ran}{Ran}

\newcommand{\ShowColoredChanges}{true} 
\ifthenelse{\isundefined{\ShowColoredChanges}}
  {

  }
  {

  }

\begin{document}

\title[Time-periodic oscillating N\'eel walls]{Time-periodic oscillating N\'eel walls in ferromagnetic thin films}

\author[A. Capella]{Antonio Capella \orcidlink{0000-0002-5397-6440}}

\address{{\rm (A. Capella)} Instituto de Matem\'aticas\\Universidad Nacional Aut\'onoma de M\'exico\\Circuito Exterior s/n, Ciudad Universitaria\\C.P. 04510 Cd. de M\'{e}xico (Mexico)}

\email{capella@matem.unam.mx}

\author[V. Linse]{Valentin Linse \orcidlink{0009-0001-5050-7227}}

\address{{\rm (V. Linse)} 
Lehrstuhl f\"ur Angewandte Analysis \\ RWTH Aachen\\D-52056 Aachen (Germany)}

\email{linse@math1.rwth-aachen.de}

\author[C. Melcher]{Christof Melcher \orcidlink{0000-0002-0535-2449}}

\address{{\rm (C. Melcher)} 
Lehrstuhl f\"ur Angewandte Analysis \\ RWTH Aachen\\D-52056 Aachen (Germany)}

\email{melcher@rwth-aachen.de}

\author[L. Morales]{Lauro Morales \orcidlink{0000-0003-3294-755X}}

\address{{\rm (L. Morales)} Instituto de Investigaciones en Matem\'aticas Aplicadas y en Sistemas\\Universidad Nacional Aut\'onoma de M\'exico\\Circuito Escolar s/n, Ciudad Universitaria\\C.P. 04510 Cd. de M\'{e}xico (Mexico)}
\curraddr{\textsc{Departamento de Matem\'aticas\\Universidad Aut\'onoma Metropolitana\\Unidad Iztapalapa\\Av. San Rafael Atlixco 186, Col. Vicentina, C.P. 09340\\Cd. de M\'{e}xico (Mexico)}}

\email{lauro\_mm@xanum.uam.mx}

\author[R. G. Plaza]{Ram\'on G. Plaza \orcidlink{0000-0001-8293-0006}}

\address{{\rm (R. G. Plaza)} Instituto de Investigaciones en Matem\'aticas Aplicadas y en Sistemas\\Universidad Nacional Aut\'onoma de M\'exico\\Circuito Escolar s/n, Ciudad Universitaria\\C.P. 04510 Cd. de M\'{e}xico (Mexico)}

\email{plaza@aries.iimas.unam.mx}

\begin{abstract}
This paper studies the existence, the structure and the spectral stability of time-periodic oscillating 180-degree N\'eel walls in ferromagnetic thin films. It is proved that time-periodic coherent structures do exist as solutions to the reduced model for the in-plane magnetization proposed by Capella, Melcher, and Otto (Nonlinearity 20 (2007), no.~11, 2519--2537) when a weak and $T$-periodic external magnetic field is applied in the direction of the easy axes of the film, perturbing in this fashion the well-known static 180-degree N\'{e}el wall. The linearization around this time-periodic N\'eel wall is constituted by a family of linear operators, parametrized by the time variable, which generates an evolution system of generators (or propagator) for the linear problem. Profiting from the stability of the static N\'eel wall, it is shown that the Floquet spectrum of the monodromy map for the propagator is contained in the complex unit circle, proving stability of the oscillating solution at least at a linear level.
\end{abstract}

\keywords{N\'eel walls, ferromagnetic thin films, spectral stability, resolvent estimates.}

\subjclass[2020]{35B35, 35Q60, 35C07, 82D40, 78M22, 47A10.}

\maketitle
\setcounter{tocdepth}{1}



\section{Introduction}
\label{secintro}

A ferromagnetic thin film is a layer of magnetic material with microscopic thickness, typically less than 50 nm, that exhibits spontaneous magnetization. These films maintain magnetic order in the absence of an external field and are crucial for applications like spintronics, data storage, and magnetic sensors. The evolution of the magnetization inside a ferromagnetic material is described by the Landau-Lifshitz-Gilbert (LLG) model (cf. \cite{LanLif35,Gilb55}), which relates the observed magnetization patterns to the result of minimizing a micromagnetic energy functional, and it is posed in terms of a damped gyromagnetic precession of a (unit) magnetization vector field. The term \emph{domain wall} refers to a narrow transition region between opposite magnetization vectors inside a ferromagnet (cf. Hubert and Sch\"afer \cite{HuSch98}) and the study of their dynamics constitute one of the most fundamental topics in the theory of ferromagnetic materials. N\'eel \cite{Neel55} observed that, when the thickness of a certain ferromagnet becomes sufficiently small with respect to the exchange characteristic length (that is, in the ferromagnetic thin film limit), it then becomes energetically favorable for the magnetization to rotate in the thin film plane, giving rise to a \emph{N\'eel wall} separating two opposite magnetization regions by an in-plane rotation oriented along an axis. In other words, the normal component of the magnetization of the material is penalized by the geometry in such a way that this component must vanish as the height of the sample does (see \cite{CMO07,DKO06,GC04,Melc10,Melc03} for further details).

Capella, Melcher and Otto \cite{CMO07} proposed an effective one-dimensional thin film reduction of the micromagnetic energy, resulting into a thin-film layer equation for the in-plane magnetization's phase. The effective equations encompass a wave-type dynamics for the N\'eel wall's phase. The authors in \cite{CMO07} proved the existence of a static wave profile solution to the resulting model in the absence of an external magnetic field describing a 180-degree N\'eel wall for the in-plane magnetization. Their model, which can be extended in order to consider external applied magnetic fields, effectively describes the dynamics of N\'eel walls in the thin film limit (see also Chermisi and Muratov \cite{ChMu13} and Muratov and Yan \cite{MuYa16} for further information). In a recent  contribution, Capella \emph{et al.} \cite{CMMP24} proved that this static N\'eel wall's phase is nonlinearly stable under small perturbations. The analysis is based on a detailed spectral study of the linearized operator around the static profile and the generation of an exponentially decaying semigroup outside a one-dimensional space associated to the translation eigenfunction.

This paper is devoted to the following question: if an external, time-periodic magnetic field is applied in the direction of the easy axes of the film, what happens to the static N\'eel wall? In order to provide an answer, we assume that the applied magnetic field is sufficiently weak, that is, there exists a physical parameter measuring the intensity of the external magnetic field and which plays the role of a perturbation variable. We show that when the intensity of the magnetic field is small enough, there exists a \emph{time-periodic oscillating N\'eel wall}, that is, a solution to the same effective in-plane magnetization model which is time-periodic, with the same fundamental period as the external magnetic field, and which oscillates in the vicinity of the static N\'eel wall. More precisely, it keeps the shape of the static N\'eel wall plus a small correction. In order to show that the constructed time-periodic solution is physically plausible, we also perform a stability analysis. Upon linearization of the model equations around this oscillating N\'eel wall, we prove that the resulting linearized operator (defined on the appropriate energy spaces) and endowed with time-periodic coefficients, generates an evolution system or propagator that encodes the dynamics of the solutions to the linear problem. We prove that the Floquet spectrum, defined as the spectrum of the monodromy map operator of the evolution system after a fundamental period, is confined to the unit circle in the complex plane. This shows the spectral stability property for the time-oscillating solution, a fundamental physical property and a necessary condition for the persistence of the coherent structure under small perturbations.

To sum up, we prove that time-periodic oscillating N\'eel walls in the thin film limit appear when a periodic external magnetic field is applied. This solution is periodic in time with the same fundamental period as the period of the applied magnetic field and keeps the shape of the static N\'eel wall up to an error term. Moreover, it is spectrally stable under small perturbations in an appropriate energy space. The result is new in the thin film limit. Up to our knowledge, there is only one result addressing the existence of time-periodic N\'eel walls: Huber \cite{Hub11} proved the existence of time-periodic N\'eel wall motions for a layer of small (but positive) thickness, a physical feature that makes the problem mathematically different, because in the case of layers with positive thickness the linearized operators around domain walls are local and sectorial, in contrast with the present case of a thin-film limit.

\subsection*{Plan of the paper}

The paper is structured as follows. The preliminary Section \ref{secNeelthin} contains a description of the reduced thin-film model, of the static N\'eel wall solution and of its main spectral properties. Section \ref{seconw} is devoted to proving that time-periodic oscillating N\'eel wall solutions do exist when a periodic weak external magnetic field is applied. Section \ref{secspectral} focuses on the linearization around the constructed time-periodic solution and on the generation of an evolution system. Finally, in Section \ref{secspecstab} it is proved that the Floquet spectrum of the linearization around the time-periodic N\'eel wall is stable.

\section{N\'{e}el walls in ferromagnetic thin films}
\label{secNeelthin}

\subsection{Notation and preliminaries}
\label{secprel}

The set of unitary vectors in $\R^n$ is denoted by $\mathbb{S}^{n-1}$ and we write $\N_0 = \N \cup \{ 0 \}$. We denote the real and imaginary parts of a complex number $\lambda$ by $\Re\lambda$ and $\Im\lambda$, respectively. For any scalar or complex-valued function, the operation $(\cdot)^*$  denotes  complex conjugation. We denote the Lebesgue and Sobolev spaces  $L^2(\R, \C)$ and $H^k(\R, \C)$, $k \in \N$, of complex-valued functions on the real line as $L^2$ and $H^k$, respectively. Their real-valued counterparts are denoted as $L^2(\R)$ and $H^k(\R)$. The operators $\widehat{\cdot}:L^2\to L^2$ and $\widecheck{\cdot}:L^2\to L^2$ stand for the Fourier transform and its inverse, respectively, and the half-Laplacian is defined by the relation $(-\Delta)^{1/2}u = (|\xi|\widehat{u})\,\widecheck{}$. The symbol ``$\lesssim$'' means ``$\leq$'' times a harmless uniform positive constant $C > 0$, i.e., $f \lesssim g$ if and only if $f \leq C g$.

Linear operators acting on infinite-dimensional spaces are indicated with calligraphic letters (e.g., $\cL$), except for the identity operator which is indicated by $I$. For given Banach spaces $X$ and $Y$, the spaces of linear, bounded and closed operators from $X$ to $Y$ are denoted as $\ccL(X,Y)$, $\ccB(X,Y)$ and $\ccC(X,Y)$, respectively. The domain of a linear operator, $\cL : X \to Y$, is denoted as $D(\cL) \subseteq X$ and its range is $\text{Ran}(\cL) \subset Y$. For any linear, closed and densely defined operator operator $\cL : D(\cL) \subset X \to Y$, the resolvent set, $\rho(\cL)$, is  defined as the set of complex numbers $\lambda\in \C$ such that $\cL -\lambda$ is injective and onto, and $(\cL - \lambda)^{-1}$ is a bounded operator. The spectrum of $\cL$ is the complex complement of the resolvent, $\sigma(\cL)=\C\setminus \rho(\cL)$. Finally, if $X$ is a Banach space and $\cA$, $\cT$ are two operators in $X$ such that $D(\cA) = D(\cT)$  then, the commutator $\llb \cA, \cT \rrb:D(\cA)\subset X\to X$ is given by the difference $\llb \cA, \cT \rrb := \cA \cT- \cT\cA $. 

\subsection{Micromagnetics, N\'{e}el walls and the thin film limit}
\label{secLLG}

The dynamics of the magnetization distribution in a ferromagnetic body, $\widetilde{\Omega} \subset \R^3$, is determined by the Landau-Lifshitz-Gilbert (LLG) model (cf. \cite{LanLif35,Gilb55}):
\begin{equation}
\label{eqLLG}
\bm_t+ \bm \times (\gamma \bH_{\mathrm{eff}} - \alpha \bm_t) = 0.
\end{equation}
Here $\bm : \widetilde{\Omega} \times (0,\infty) \to \mathbb{S}^2 \subset \R^3$ is the magnetization vector field, $\alpha > 0$ is a non-dimensional damping coefficient known as the Gilbert factor, $\gamma > 0$ is the constant value of the gyromagnetic ratio with dimensions of frequency (cf. Gilbert \cite{Gilb04}) and $\bH_{\mathrm{eff}} = \bh - \nabla \E(\bm)$ is the effective magnetic field, that is, the net magnetic field acting on the magnetization $\bm$. It consists of the applied field $\bh$ and the negative functional gradient of the micromagnetic interaction energy $\E(\bm)$ which, in the absence of external fields, is given by
\begin{equation}
\label{energyint}
 \E(\bm) = \frac{1}{2}\Big( d^2 \int_{\widetilde{\Omega}} |\nabla \bm|^2 \, dx + \int_{\R^3} |\nabla U|^2 + Q \int_{\widetilde{\Omega}} \Phi(\bm) \, dx \Big).
\end{equation}
The constant parameter $d> 0$ is the exchange length and the stray field, $\nabla U$, is defined uniquely via $\Delta U = \textrm{div}\,(\bm \chi_{\widetilde{\Omega}})$, where $\chi_G$ denotes the indicator function of the set $G$. Crystalline anisotropies are encoded in the last integral of \eqref{energyint} through a penalty energy density $\Phi$ which usually has the form of an even polynomial in $\bm \in \mathbb{S}^2$. The constant $Q>0$ measures the relative strength of anisotropy penalization against stray-field interaction.

We are interested in the thin-film regime when $\widetilde{\Omega} = \Omega \times (0,\delta)$ with $\Omega \subset \R^2$ and $0 < \delta \ll d$. In this case it can be assumed that the magnetization is independent of $x_3$ and $\ell$-periodic in the ${\bf e}_2$ direction (see, e.g., Garc\'{\i}a-Cervera \cite{GC04}, Melcher \cite{Melc10} and De Simone \emph{et al.} \cite{DKO06}). This yields,
\[
{\bf m}(x_1,x_2+\ell) = \bm(x_1,x_2), \quad\text{for any } x = (x_1,x_2)\in\R^2.
\]
Since the material exhibits uniaxial anisotropy in the ${\bf e}_2$ direction with $\Phi({\bf m}) = 1- m_2^2$, it can be shown that, under the appropriate scalings, ${\bf m} = (m,0)$ with $m = (m_1,m_2)$ is a solution to the following variational problem in one space dimension,
 \begin{equation}
 \label{limit-varppio}
 \begin{aligned}
E_0 (m) &= \tfrac{1}{2} \left({\widetilde{Q}} 
\Vert m'\Vert_{L^2(\R)}^2 + \Vert m_1 \Vert_{\dot{H}^{1/2}(\R)}^2 + \Vert m_1\Vert_{L^2(\R)}^2    \right) \to \min,\\ 
 &m : \R \to {\mathbb S}^1, \qquad \text{with} \;\;  
 m(\pm\infty) =(0,\pm 1),
 \end{aligned}
\end{equation}
where $' = d/d x_1$ and $\widetilde{Q} > 0$ is a rescaled constant depending on $Q$. It is to be observed that $E_0(m)$ is a strictly convex functional on $m_1$ inasmuch as $|m'|^2=(m_1')^2/(1-m_1^2)$. Thus, the variational problem \eqref{limit-varppio} has a minimizer for any $\widetilde{Q}>0$. Since the left translation operator is an isometry in $L^2$, the expression of $E_0(m)$ is invariant under spatial translations and this invariance is inherited by the energy, yielding that minimizers of \eqref{limit-varppio} are unique up to translations. This property of the energy and its minimizres is known as \emph{translation invariance}. The minimizer that satisfies $m_1(0)=1$ is called the {\em N\'eel wall profile}. We refer to $E_0(m)$ as the N\'eel wall energy.

Capella \emph{et al.} \cite{CMO07} proved that, in the thin film limit regime, the minimizers of \eqref{limit-varppio} are solutions to an effective geometric nonlinear wave equation for the in-plane magnetization $m = (m_1, m_2)$, which reads
\begin{equation}
\label{nlwm}
\begin{aligned}
\big[ \partial_t^2 m + \nu \partial_t m + \nabla E_0(m) \big] &\perp T_m \mathbb{S}^1,\\
m : \R \times (0,\infty) &\to\mathbb{S}^1,\\
m(\pm \infty, t) &= (0, \pm 1).
\end{aligned}
\end{equation}
In this fashion, the in-plane magnetization is completely determined by its phase, $\theta \in (-\pi/2,\pi/2)$, through the relation $m = (m_1, m_2) = (\cos \theta, \sin \theta)$. The variational problem that defines a N\'eel wall is recast as the following variational problem for the phase,
\begin{equation}
 \label{varprob}
 \begin{aligned}
 \mathcal{E}(\theta) &= \tfrac{1}{2} \big( \|\theta'\|_{L^2}^2 + \|\cos \theta\|^2_{{\dot H}^{1/2}} + \|\cos \theta\|_{L^2}^2 \big) \; \rightarrow \;  \min\\ 
 \theta : \R &\to (-\pi/2,\pi/2), \qquad \text{with} \;\; \theta(\pm\infty) = \pm \pi/2
 \end{aligned}
\end{equation}
(for details, see Capella \emph{et al.} \cite{CMO07}). Notice that the energy $\mathcal{E}(\theta)$ is exactly the same energy $E_0$ appearing in \eqref{limit-varppio}, but recast as a function of the phase, $\mathcal{E}(\theta) = E_0(\cos \theta, \sin \theta)$. For concreteness and without loss of generality, we have assumed that $\widetilde{Q} \equiv 1$. We keep such normalization for the rest of the paper. It can be proved that the phase itself also exhibits some wave-type dynamics and it is a solution to the following nonlinear wave-type limit equation,
\begin{equation}
 \label{reddyneqhom}
 \left\{ \ \ 
\begin{aligned}
&\partial_t^2 \theta + \nu \partial_t \theta + \nabla \cE(\theta) = 0, \\
&\theta(-\infty,t) =-\pi/2,\quad \theta(\infty,t) =\pi/2.
\end{aligned}
\right.
\end{equation}

The magnetization and its phase are determined by the presence of an applied external magnetic field. In the case where this external field points towards one of the end-states determined by the anisotropy, $\mathbf{h} = H \mathbf{e}_2$ (or in other words, when it is applied in the direction of the easy axes of the film), then it can be proved (cf. \cite{CMO07}) that the dynamical equation for the phase $\theta$ is given by
\begin{equation}
 \label{reddyneq}
 \left\{ \ \ 
\begin{aligned}
&\partial_t^2 \theta + \nu \partial_t \theta + \nabla \cE(\theta) = H\cos\theta, \\
&\theta(-\infty,t) =-\pi/2,\quad \theta(\infty,t) =\pi/2,
\end{aligned}
\right.
\end{equation}
where $H \in \R$ measures the external magnetic field strength and  $\cE(\theta)$ is the effective energy appearing in \eqref{varprob}. In this paper we are interested in the case when the applied external magnetic field is a periodic function of time, $H = H(t) = H(t + T)$ for all $t$ and for some fundamental period $T > 0$.

Equations \eqref{varprob} and \eqref{reddyneq}, which were rigorously derived from electromagnetic theory and the LLG equation by Capella, Melcher and Otto \cite{CMO07}, constitute an effective model that describes the dynamics of the magnetization in a ferromagnetic thin film. 

\subsection{The static N\'{e}el wall and its stability properties}
\label{secstatic}

In the absence of an external magnetic field, Capella \emph{et al.} \cite{CMO07} proved the existence of a unique odd, monotone increasing and smooth steady state, 
\begin{equation}
\label{staticNeelprof}
\theta_0 = \theta_0(x),
\end{equation}
such that $\partial_x \theta_0 > 0$, $\theta_0(\pm \infty) = \pm \tfrac{\pi}{2}$, known as the \emph{static N\'eel wall's phase profile}. This solution is unique up to translations and it is a minimizer of the variational problem \eqref{varprob}. 
%
%
%
%
The main structural properties of the static N\'eel wall can be summarized as follows.
\begin{proposition}[properties of the static N\'eel wall's phase \cite{CMMP24,CMO07,Melc03}]
\label{propNeelw}
 There exists a static N\'eel wall solution with phase $\theta_0 = \theta_0(x)$, $\theta_0 : \R \to (-\pi/2,\pi/2)$, satisfying the following:
 \begin{itemize}
  \item[\rm{(a)}] 
  \label{propa} $\theta_0$ is a strict minimizer of the variational problem \eqref{varprob}, with center at the origin, $\theta_0(0) = 0$, and monotone increasing, $\partial_x \theta_0 (x) > 0$ $\, \forall x \in \R$.
  \item[\rm{(b)}]\label{propb} $\theta_0$ is a smooth solution to
  \begin{equation}\label{ELeq}
   \partial_x^2 \theta + \sin \theta (1+(-\Delta)^{1/2}) \cos \theta = 0,
  \end{equation}
which is the Euler-Lagrange equation for the variational problem \eqref{varprob}.
\item[\rm{(c)}]\label{propc} $\partial_x \theta_0  \in H^k(\R)$ for any $k \in \Z$, $k \geq 0$.
\item[\rm{(d})]\label{propd} For every $u \in H^1$ such that $u(0) = 0$ there holds
\begin{equation}
 \label{Hesspos}
 \Hess \mathcal{E}(\theta_0) \langle u,u \rangle_{L^2} \geq \|u \, \partial_x \theta_0\|^2_{L^2} + \Re \, b [u\sin \theta_0, u \sin \theta_0],
\end{equation}
where the bilinear form $b[\cdot,\cdot] : H^1 \times H^1 \to \C$, defined as,
\begin{equation}
\label{defbilinearB}
 b[f,g] = \int_\R (1+|\xi|) \hat f(\xi) \hat g(\xi)^* \, d\xi, \qquad f, g \in H^1,
\end{equation}
is equivalent to the standard inner product in $H^{1/2}$.
\item[\rm{(e)}]\label{prope} $\theta_0 \in W^{2,\infty}$.
\end{itemize}
\end{proposition}
\begin{proof}
See Lemmata 1 and 2, as well as Theorem 1 in \cite{CMO07}. See also Propositions 1 and 2 in \cite{Melc03}, and Proposition 2.1 and Corollary 2.2 in \cite{CMMP24}.
\end{proof}

\begin{remark}
This static solution $\theta_0 = \theta_0(x)$ is a minimizer of the variational problem \eqref{varprob}. Therefore, $\nabla \mathcal{E}(\theta_0) = 0$ and it can be interpreted as a \emph{stationary} traveling wave solution to \eqref{reddyneqhom}, monotone increasing and connecting $- \pi/2$ at $x=-\infty$ with $\pi/2$ at $x = \infty$. Moreover, since the energy $\mathcal{E}(\theta)$ coincides with the N\'eel wall energy $E_0(m)$, it is clear that this phase defines an in-plane magnetization, $m_0(x) = (\cos \theta_0, \sin \theta_0)$, which is a static solution to the wave-type equation \eqref{nlwm} and a minimizer of the variational problem \eqref{limit-varppio}.
\end{remark}

The main result in Capella \emph{et al.} \cite{CMMP24} is that this static N\'eel wall is nonlinearly stable with respect to small perturbations. To that end, the linearized operator around the static N\'eel wall's phase plays a fundamental role. The latter is given by (see \cite{CMO07,CMMP24})
\begin{equation}
\label{defL0}
\left\{
\begin{aligned}
\cL_0 &: L^2 \to L^2,\\
D(\cL_0) &= H^2,\\
\cL_0 u &:= - \partial^2_xu + \mathcal{S}_{\theta_0}u - c_{\theta_0} u, \qquad u \in D(\cL_0),
\end{aligned}
\right.
\end{equation}
where the non-local, linear operator $\mathcal{S}_{\theta_0}$ is defined as
\begin{equation}
\label{defS0}
\left\{
\begin{aligned}
\mathcal{S}_{\theta_0} &: L^2 \to L^2,\\
D(\mathcal{S}_{\theta_0}) &= H^1,\\
\mathcal{S}_{\theta_0} u &:= \sin \theta_0 (1+(-\Delta)^{1/2}) (u \sin \theta_0), \qquad u \in D(\mathcal{S}_{\theta_0}),
\end{aligned}
\right.
\end{equation}
and
\[
c_{\theta_0}(x) := \cos \theta_0(x) (1+(-\Delta)^{1/2}) \cos \theta_0(x),
\]
is a bounded coefficient. For simplicity we write $s_{\theta_0} = s_{\theta_0}(x)$ to denote the function $x \mapsto \sin \theta_0(x)$ for all $x \in \R$. 

It is to be observed that the operator $\cL_0$ is the linearization of $\nabla \cE$ around the static phase $\theta_0$, namely,
\[
\cL_0 = D^2 \cE(\theta_0) : L^2 \to L^2
\]
(for details, see \cite{CMO07,CMMP24}). The following Proposition summarizes its main spectral properties.

\begin{proposition}[spectral properties of $\cL_0$ \cite{CMO07,CMMP24}]
\label{propspecL0}
The operator $\cL_0 : L^2 \to L^2$ defined in \eqref{defL0} satisfies the following:
\begin{itemize}
\item[\rm{(a)}] $\cL_0$ is a closed, densely defined, self-adjoint linear operator.
\item[\rm{(b)}] There exists a positive constant $\Lambda_0 > 0$ such that
\[
\sigma(\cL_0) \subset \{0\} \cup [\Lambda_0, \infty).
\]
\item[\rm{(c)}] $\lambda = 0$ is a simple eigenvalue of $\cL_0$, associated to the eigenfunction $\partial_x \theta_0 \in H^2$.
\item[\rm{(d)}] Let $L^2_\perp$ denote the $L^2$-orthogonal complement of $\text{span} \{ \partial_x \theta_0 \}$ and let $H_\perp^k := H^k \cap L^2_\perp$, for all $k \in \N$. Then for all $u \in H_\perp^k$ there holds
\[
\langle \cL_0 u, u \rangle_{L^2} \geq \Lambda_0 \| u \|_{L^2}^2.
\]
\end{itemize}
\end{proposition}
\begin{proof}
See Lemma 3 in \cite{CMO07}, as well as Theorem 4.1, Proposition 4.6 and Lemmata 4.9 and 4.10 in \cite{CMMP24}.
\end{proof}

If $\theta_0 + u$ is a solution to the nonlinear wave equation \eqref{reddyneqhom} where $u$ now denotes a perturbation, then the linearized equation around $\theta_0$ takes the form
\[
\partial_t^2 u + \nu \partial_t u + \cL_0 u = 0.
\]
This is a linear, second order wave equation for the perturbation $u$, which requires initial conditions for both $u$ and $\partial_t u$. Performing the standard change of variables in order to write it as a first order system, set $v = \partial_t u$ and recast it as
\[
\partial_t  \begin{pmatrix} u \\ v \end{pmatrix} = \begin{pmatrix} 0 & I \\ - \cL_0 & \nu I  \end{pmatrix} \begin{pmatrix} u \\ v \end{pmatrix}.
\]
Hence, we define the block matrix operator $\cA_0$ as
\begin{equation}
\label{defA0}
\left\{
\begin{aligned}
\cA_0 &: H^1 \times L^2 \to H^1 \times L^2,\\
D(\cA_0) &= H^2 \times H^1,\\
\cA_0 &:= \begin{pmatrix} 0 & I \\ - \cL_0 & \nu I  \end{pmatrix}.
\end{aligned}
\right.
\end{equation}
\begin{remark}
\label{remarkA0}
$\cA_0$ is a closed, densely defined operator in $H^1 \times L^2$ (cf. \cite{CMMP24}). Here, the operator $\cL_0$ appearing in \eqref{defA0} refers to the restriction of $\cL_0$ to $H^1$, namely, to ${\cL_0}_{|H^1}$, which we write once again as $\cL_0$ with a slight abuse of notation. Nonetheless, its spectral properties (see Proposition \ref{propspecL0}) remain the same; see also Remark 5.2 in \cite{CMMP24}.
\end{remark}

The following Proposition summarizes the main properties of the block operator $\cA_0$, which will be useful later on.

\begin{proposition}[spectral properties of $\cA_0$ \cite{CMMP24,CMMP25}]
\label{propA0}
Let $\cA_0$ be the block operator defined in \eqref{defA0}. Then it satisfies the following:
\begin{itemize}
\item[\rm{(a)}] $\lambda = 0$ is a simple and isolated eigenvalue of $\cA_0$, associated to the eigenfunction
\begin{equation}
\label{defThet0}
\Theta_0 := (\partial_x \theta_0, 0) \in D(\cA_0) = H^2 \times H^1.
\end{equation}
\item[\rm{(b)}] There exists $\zeta_0(\nu) > \tfrac{1}{2}\nu > 0$ such that
\begin{equation}
\label{spectA0}
\sigma(\cA_0) \subset \{ 0 \} \cup \{ \lambda \in \C \, : \, \Re \lambda \leq - \zeta_0(\nu) < 0 \}.
\end{equation}
\item[\rm{(c)}] $\cA_0$ is the infinitesimal generator of a $C_0$-semigroup of quasi-contractions, $\{ e^{t \cA_0} \}_{t \geq 0} \subset \ccB(H^1 \times L^2)$, and there exists $\omega \in \R$ such that
\[
\big\| e^{t \cA_0}\big\|_{H^1 \times L^2 \to H^1 \times L^2} \leq e^{\omega t},
\]
for all $t \geq 0$.
\end{itemize}
\end{proposition}
\begin{proof}
See Theorem 2.4 in \cite{CMMP25}, and Lemmata 5.6 and 5.5 in \cite{CMMP24} for the proof.
\end{proof}

\section{Existence and structure of time-periodic N\'{e}el walls}
\label{seconw}

This Section is devoted to show the existence of time-periodic solutions to the Landau-Lifshitz-Gilbert model in the thin film limit in the presence of a time-periodic external applied magnetic field. The effective model under consideration has the form,
\begin{equation}
 \label{LLG-H}
 \left\{ \,
\begin{aligned}
&\partial_t^2 \theta + \nu \partial_t \theta + \nabla \cE(\theta) = \varepsilon H(t) \cos\theta, \\
&\theta(-\infty,t) =-\pi/2,\quad \theta(\infty,t) =\pi/2,
\end{aligned}
\right.
\end{equation}
where $\theta = \theta(x,t)$ denotes the N\'eel wall's phase, $\varepsilon > 0$ is a scalar parameter that measures the external magnetic field strength, and  $\cE(\theta)$ is the effective energy appearing in \eqref{varprob}. In this model, $H = H(t)$ represents a continuously differentiable time-periodic magnetic field, with constant fundamental period $T > 0$. More precisely, we assume that the external magnetic field is \emph{$T/2$-symmetric}, namely
\begin{equation}
\label{T2symm}
H(t) = - H(t + T/2),
\end{equation}
for all $t$. Relation \eqref{T2symm} clearly implies that the function $H$ is $T$-periodic and has zero mean over a period,
\[
\int_0^T H(t) \, dt = 0.
\]

\begin{remark}
\label{remT2symm}
Materials with intrinsic magnetism, like ferromagnets below their Curie temperature, have a net magnetic moment which inherently breaks time-reversal symmetry. Indeed, the presence of an ordered magnetic state introduces a fixed directional dependence, and reversing time would reverse the spin currents (a moving charged particle's direction of deflection becomes asymmetrical in time). Hence, the assumption of a $T/2$-symmetric external magnetic field is compatible with materials with broken time-reversal symmetry, allowing for a specific direction of time to be preferred (see, e.g., \cite{dBLB19}). Applications include both the superconducting diode \cite{NFW23} and the crystal Hall effects \cite{SGJS20}.
%
%
%
%
\end{remark}

Before stating the main existence result, it is convenient to define the function spaces to work on.

\begin{definition}
\label{defspaces}
    Set $Q_T = \mathbb{T}_T\times \R$, where $\mathbb{T}_T$ is the flat torus of length $T > 0$. For $k,l \in \N_0$ we define the function spaces
    \begin{align*}
        C^k_tH^l_x(Q_T) &:= C^k(\mathbb{T}_T;H^l(\R)) \\
        C^k_tH^l_\perp(Q_T) &:= \{u \in C^k_tH^l_x(Q_T) \: | \: \pld{u(\cdot, s)}{\partial_x \theta_0(\cdot)} = 0 \text{ for all } s \in \bb{T}_T\},
    \end{align*}
    as well as the linear operator 
    \begin{gather*}
        \cT_0:D(\cT_0) \to C^0_tH^1_\perp(Q_T), \\
        \cT_0u := \partial_t^2 u + \nu \partial_t u + \mathcal{L}_0 u,
    \end{gather*}
    with domain 
    \begin{gather*}
        D(\cT_0) := \{u \in (C^0_tH^2_\perp \cap C^1_tH^1_\perp \cap C^2_t L^2_\perp)(Q_T) \: | \: (\partial_t^2 - \partial_x^2)u \in C^0_tH^1_\perp(Q_T)\}.
    \end{gather*}
\end{definition}

\begin{remark}
    The space $D(\cT_0)$ is complete with respect to the graph norm $\|\cdot\|_{D(\cT_0)}$ by Proposition \ref{T_0_closed} below.
\end{remark}

The main result of this Section can be stated as follows.

\begin{theorem}[existence of the time-periodic oscillating N\'eel wall]
    \label{theorem_existence}
    Let $H(t)$ be a bounded and continuously differentiable $T/2$-symmetric applied field satisfying \eqref{T2symm}. For a sufficiently small amplitude $\varep > 0$, there exists a solution $\bar\theta = \bar\theta(x,t)$ to the reduced Landau-Lifshitz-Gilbert dynamics equation \eqref{LLG-H}
    connecting antipodal states at infinity $\bar\theta(\pm\infty,t) = \pm \frac{\pi}{2}$ for all $t \in \bb{T}_T$, such that 
    \begin{align}
        \label{def_theta}
        \bar\theta(x,t) = \theta_0(x +X(t)) + \chi(x,t),
    \end{align}
    where $X \in C^2$ is $T/2$-symmetric and $\chi(x,t) = -\chi(-x,t +T/2)$ is $T$-periodic in time. Moreover $\left \|\chi \right\|_{D(\cT_0)} = O(\varep)$ and $X = \varep Y + o(\varep)$. Here $Y$ satisfies
    \begin{align}
        \label{equation_Y}
        \ddot{Y} + \nu \dot{Y} = 2\nld{\partial_x \theta_0}^{-2}H,
    \end{align}
    where $\dot{Y} = dY/dt$.
\end{theorem}

\begin{remark}
Theorem \ref{theorem_existence} guarantees the existence of a time-periodic N\'eel wall, $\brt = \brt(x,t)$, solution to \eqref{LLG-H} when a small amplitude external $T$-periodic magnetic field is applied. This solution represents an order $O(\vep)$-perturbation of the well-known static N\'eel wall, which is periodically time-oscillating with the same fundamental period as the period of the applied magnetic field. The solution is a real function, which oscillates and keeps the shape of the static N\'eel wall up to an error term, namely $\chi(x,t)$, which is small. The motion can be understood as a time-periodic translation $X(t)$ of the static phase $\theta_0(x)$.
\end{remark}

\subsection{The linearized problem}

Prior to the analysis of the nonlinear equation \eqref{LLG-H} it is necessary to examine its linear part which coincides with the linearization around the static N\'eel wall $\theta_0$. The main result concerning this linearized problem is the following result.

\begin{theorem}
    \label{dynamical_linearized}
    Let $f \in C^0_tH^1_\perp(Q_T)$. Then the equation
    \begin{align*}
        \partial_t^2 u + \nu\partial_t u + \mathcal{L}_0 u = f
    \end{align*}
    admits solution $u \in (C^2_tL^2_x\cap C^1_tH^1_x \cap C^0_tH^2_x) (Q_T)$ unique up to a multiple of $\partial_x \theta_0$.
\end{theorem}

We present a somewhat self-contained proof,  following a different argumentation than the one given in \cite{CMMP24}. The proof relies on the following lemmata.

\begin{lemma}\label{lemma_X_hilbert}
    The (real) space $Z := H^1_\perp(\R) \times L^2_\perp(\R)$ equipped with the inner product 
    \begin{align*}
        \skalprod{(u,v)}{(\widetilde{u},\widetilde{v})}_Z := \skalprod{u}{\mathcal{L}_0 \widetilde{u}}_{L^2} + \skalprod{v}{\widetilde{v}}_{L^2}
    \end{align*}
    is a Hilbert space.
\end{lemma}
\begin{proof}
    By Proposition \ref{propspecL0} (d), it is known that $\skalprod{u}{\mathcal{L}_0 u}$ is equivalent to $\norm{u}_{H^1}^2$ on $H^1_\perp(\R)$. Therefore, $\norm{(u,v)}_Z$ is equivalent to $\norm{u}_{H^1} + \norm{v}_{L^2}$ on the subspace $Z \subset H^1(\R)\times L^2(\R)$. Since $Z = F^{-1}(\{0\})$ for the bounded linear functional $F:H^1\times L^2 \to \R^2, F(u,v) = \left(\skalprod{u}{\partial_x \theta_0}_{L^2},\skalprod{v}{\partial_x \theta_0}_{L^2}\right)$, $Z$ is indeed a Banach space, as it is a closed subspace of the Hilbert space $(H^1 \times L^2, \norm{\cdot}_{H^1\times L^2})$. Since the norm is induced by a scalar product, $Z$ is a Hilbert space.
\end{proof}

\begin{lemma}\label{lemma_A_closed}
    The linear block-operator $\cA_\perp : D(\cA_\perp) := H^2_\perp(\R)\times H^1_\perp(\R) \subset Z \to Z$ defined by 
    \begin{align}
        \cA_\perp \begin{pmatrix} u \\ v \end{pmatrix} := \begin{pmatrix}
            0 & I \\
            -\mathcal{L}_0 & -\nu I
        \end{pmatrix}\begin{pmatrix}
            u \\ v
        \end{pmatrix}
    \end{align}
    is closed.
\end{lemma}

\begin{remark}
	\label{remarkAperp}
	The operator $\mathcal{A}_\perp$ is nothing but the restriction of $\mathcal{A}_0$ onto $D(A_\perp) \subset \langle \Theta_0 \rangle^\perp$. 
\end{remark}

\begin{proof}
    This is a direct consequence of Remarks \ref{remarkA0} and \ref{remarkAperp} and the fact that $D(\mathcal{A}_\perp) \subset D(\mathcal{A}_0)$ is closed with respect to the graph norm induced by $\mathcal{A}_0$.  $D(\mathcal{A}_\perp)$ equipped with the graph norm is thus complete,  and $\mathcal{A}_\perp:D(\mathcal{A}_\perp) \to Z$ continuous.
\end{proof}

\begin{lemma}\label{lemma_D(A)_dense}
    The domain $D(\cA_\perp) = H^2_\perp(\R) \times H^1_\perp(\R)$ is a dense subset of $H^1_\perp(\R)\times L^2_\perp(\R)$.
\end{lemma}
\begin{proof}
   $D(\mathcal{A}_\perp)$ coincides with $D(\mathcal{A}_0)\cap F^{-1}(\{0\})$ for the bounded linear functional $F:H^1(\R)\times L^2(\R) \to \R^2$, $(u,v) \mapsto (\skalprod{u}{\theta_0'}_{L^2},\skalprod{v}{\theta_0'}_{L^2})$. Density of $D(\mathcal{A}_0)$ in $H^1(\R)\times L^2(\R)$ (see remark \ref{remarkA0}) and equivalence of the norms $\norm{(u,v)}_Z$ and $\norm{u}_{H^1} + \norm{v}_{L^2}$ on this very subspace (see Proposition \ref{propspecL0}) yield the claim. 
\end{proof}

\begin{lemma}\label{lemma_A_resolventset}
    The positive real line is contained in the resolvent set of $\cA_\perp$, i.e. 
    \begin{align*}
        (0,\infty) \subset \rho(\cA_\perp).
    \end{align*}
\end{lemma}

\begin{proof}
    Let $\lambda > 0$ be any positive number, $(f,g) \in Z$ and examine the equation
    \begin{align}
        \label{resolvent_equation}
        (\lambda I - \cA_\perp)(u,v) = (f,g).
    \end{align}
    We want to show that equation (\ref{resolvent_equation}) admits a unique solution $(u,v) \in D(\cA_\perp)$.
    Equation (\ref{resolvent_equation}) corresponds to the system of equations
    \begin{gather}
        \lambda u - v = f, \quad \text{and} \label{first_equation} \\
        (\lambda + \nu)v + \mathcal{L}_0 u = g. \label{second_equation}
    \end{gather}
    Inserting (\ref{first_equation}) into (\ref{second_equation}), we obtain 
    \begin{align}
        \label{inserted_equation}
        [(\lambda + \nu)\lambda I + \mathcal{L}_0] u = g + (\lambda + \nu)f.
    \end{align}
    Since $\nu,\lambda > 0$ and $f,g \in L^2_\perp$, there does indeed exist a unique solution $u \in H^2_\perp(\R)$ to the above equation: $\mathcal{L}_0$ is self-adjoint and $\skalprod{u}{\mathcal{L}_0 u}_{L^2} \geq \Lambda_0\norm{u}_{H^1(\R)}^2$ for $u \in H^2_\perp(\R)$. Thus, the spectrum of $-\mathcal{L}_0$ is real and bounded from above by $-\Lambda_0 < 0$ (see Proposition \ref{propspecL0}). This spectral bound is negative, implying $0 < (\lambda(\lambda + \nu)) \in \rho(-\mathcal{L}_0)$. $u\in H^2_\perp$ is thus uniquely determined by $g$ and $f$. Thus, $v = \lambda u - f \in H^1_\perp(\R)$ is also uniquely defined. Hence, the resolvent set of $\cA_\perp$ satisfies $(0,\infty) \subset \rho(\cA_\perp)$, as claimed.
\end{proof}

\begin{lemma}\label{lemma_A_resolvent_operator}
    Let $\lambda > 0$. The resolvent operators $\cR(\lambda; \cA_\perp) := (\lambda I - \cA_\perp)^{-1}$ satisfy the bound 
    \begin{align*}
        \norm{\cR(\lambda; \cA_\perp)}_{Z \to Z} \leq \lambda^{-1}
    \end{align*}
    in operator norm.
\end{lemma}

\begin{proof}
    Testing equation (\ref{second_equation}) by $v$ yields
    \begin{align}
        \label{multiplied_equation}
        (\lambda + \nu)\norm{v}_{L^2}^2 + \skalprod{v}{\mathcal{L}_0 u}_{L^2} = \skalprod{v}{g}_{L^2}.
    \end{align}
    Inserting (\ref{first_equation}) into (\ref{multiplied_equation}), we obtain
    \begin{gather}
        (\lambda + \nu)\norm{v}_{L^2}^2 + \lambda\skalprod{u}{\mathcal{L}_0 u}_{L^2} = \skalprod{v}{g}_{L^2} + \skalprod{f}{\mathcal{L}_0 u}_{L^2} \\
        \leq (\norm{g}_{L^2}^2 + \skalprod{f}{\mathcal{L}_0 f}_{L^2})^{1/2}(\norm{v}_{L^2}^2 + \skalprod{u}{\mathcal{L}_0 u}_{L^2})^{1/2}, \nonumber
    \end{gather}
    where we used the Cauchy-Schwarz inequality and added non-negative terms inside the square roots. Since $\nu \norm{v}^2_{L^2} > 0$, we obtain 
    \begin{align*}\label{resolvent_A_bound}
        \lambda \left(\norm{v}^2_{L^2} + \skalprod{\mathcal{L}_0u}{u}_{L^2} \right) \leq (\norm{g}_{L^2}^2 + \skalprod{f}{\mathcal{L}_0 f}_{L^2})^{1/2}(\norm{v}_{L^2}^2 + \skalprod{u}{\mathcal{L}_0 u}_{L^2})^{1/2}.
    \end{align*}
    By $\cR(\lambda;\cA_\perp)(f,g) = (u,v)$, we may express inequality (\ref{resolvent_A_bound}) as
    \begin{align}
        \lambda \norm{R(\lambda;A)(f,g)}_Z \leq \norm{(f,g)}_Z.
    \end{align}
    Thus, $\norm{\cR(\lambda;\cA_\perp)} \leq \frac{1}{\lambda}$ in the operator norm.
\end{proof}

\begin{corollary}\label{lemma_A_generator}
    The operator $\cA_\perp : D(\cA_\perp) \to Z$ is the generator of a strongly continuous semigroup of contractions $\tcS(t) :Z \to Z$.
\end{corollary}

\begin{proof}
    The proof is just an application of Hille-Yosida, the requirements of which were proven in the preceding lemmata.
\end{proof}

\begin{corollary}
\label{corollary_homogeneous_sol}
    The unique solution $u \in (C^2_tL^2_\perp \cap C^1_tH^1_\perp \cap C^0_t H^2_\perp)(Q_T)$ to the Cauchy-problem 
    \begin{align}\label{equation_homogeneous}
        \partial_t^2 u + \nu \partial_t u + \mathcal{L}_0 u = 0,
    \end{align}
    with initial data $u(0) = u_0 \in H^1_\perp(\R),\partial_t u(0) = \partial_t u_0 \in L^2_\perp(\R)$ is given by $\tcS(t)(u_0,\partial_t u_0)$.
\end{corollary}

\begin{lemma}[hypocoercivity of $\cA_\perp$]
    \label{exponential_decay}
    Let $u \in (C^2_tL^2_\perp \cap C^1_tH^1_\perp \cap C^0_tH^2_\perp) (Q_T)$ be the unique solution in Corollary \ref{corollary_homogeneous_sol}. Then $u$ satisfies the inequality
    \begin{align}
        \label{hypocoercive}
        \norm{(u(t),\partial_t u(t))}_Z^2 \lesssim e^{-Ct}\norm{(u_0,\partial_t u_0)}_{Z}^2,
    \end{align}
    for $C > 0$ depending only on $\nu$ and the coercivity constant of $\mathcal{L}_0$.
\end{lemma}

\begin{remark}\label{remark_hypocoercivity_S_norm}
    We may equivalently state, that $\tcS(t)$ satisfies the bound 
    \begin{align*}
        \big\| \tcS(t)\big\|_{Z \to Z} \lesssim e^{-C t},
    \end{align*}
    with a new uniform constant $C$.
\end{remark}

\begin{proof}[Proof of Lemma \ref{exponential_decay}]
    Multiplying equation (\ref{equation_homogeneous}) by $2 \partial_t u$ and by $\nu u$ we obtain the two equalities
    \begin{equation}\label{f_g_1}
    \begin{aligned}
        0 = & 2\skalprod{\partial_t u}{\partial_t^2 u}_{L^2}  + 2\nu \norm{\partial_t u}^2_{L^2} + 2\skalprod{\partial_t u}{\mathcal{L}_0 u}_{L^2} \\
        = & \frac{d}{dt}(\norm{\partial_t u}_{L^2} + \skalprod{u}{\mathcal{L}_0 u}_{L^2}) + 2\nu \norm{\partial_t u}^2_{L^2}
    \end{aligned}  
    \end{equation}
    and
    \begin{equation}\label{f_g_2}
    \begin{aligned}
        0 &= \nu\skalprod{u}{\partial_t^2 u}_{L^2}  + \nu^2 \skalprod{u}{\partial_t u}_{L^2} + \nu\skalprod{u}{\mathcal{L}_0 u}_{L^2} \\
        &=  \frac{d}{dt}(\frac{\nu^2}{2}\norm{u}^2_{L^2} + \nu \skalprod{u}{\partial_t u}_{L^2}) - \nu \norm{\partial_t u}^2_{L^2} + \nu \skalprod{u}{\mathcal{L}_0 u}_{L^2}. 
    \end{aligned}  
    \end{equation}
    Define $g(t) := \nu(\norm{\partial_t u}^2_{L^2} + \skalprod{u}{\mathcal{L}_0 u}_{L^2})$ and $f(t) := \norm{\partial_t u}^2_{L^2} + \skalprod{u}{\mathcal{L}_0 u}_{L^2} + \frac{\nu^2}{2}\norm{u}^2_{L^2} + \nu \skalprod{u}{\partial_t u}_{L^2}$. Then adding (\ref{f_g_1}) and (\ref{f_g_2}), we obtain 
        \begin{align}
        \label{groenewall}
        f'(t) = -g(t).
    \end{align}
    $f$ may be estimated from above by using Young's inequality and the coercivity of $\mathcal{L}_0$:
    \begin{equation}\label{f_estimate}
    \begin{aligned}
        f(t) \leq & \norm{\partial_t u}^2_{L^2} + (1 + c)\nu^2\skalprod{u}{\mathcal{L}_0 u}_{L^2} 
        \leq  C^{-1}g(t).
    \end{aligned} 
    \end{equation}
    Inserting (\ref{f_estimate}) into (\ref{groenewall}) yields the differential inequality
    \begin{align*}
        \frac{df}{dt} = -g(t) \leq -C f(t).
    \end{align*}
    By Gronwall's lemma, $f$ is bounded from above by the solution to $dv/dt = -C v$, which yields
    \begin{align*}
        f(t) \leq f(0)\exp\Big(-\int_0^t C \: ds\Big) = f(0)e^{-C t}.
    \end{align*}
    Since 
    \begin{align*}
        \norm{(u,\partial_tu)}_Z^2 \leq \skalprod{u}{\mathcal{L}_0 u}_{L^2}  + \norm{\partial_t u}^2_{L^2} + \norm{\partial_t u + \nu u}^2_{L^2} =2 f,
    \end{align*}
    and 
    \begin{align*}
        f(0) \lesssim g(0) = \nu\norm{u(0),\partial_t u(0)}_Z^2,
    \end{align*}
    as seen in (\ref{f_estimate}), our proof is complete.
\end{proof}

\begin{lemma}\label{lemma_S_H1_bound}
    Let $u$ be the solution obtained in Lemma \ref{exponential_decay} with initial data $u(0) = u_0\in H^2_\perp, \partial_t u(0) = \partial_tu_0 \in H^1_\perp$. Then it holds 
    \begin{align*}
        \norm{(\partial_x u,\partial_t \partial_x u)}_Z^2 \lesssim (1 + t)e^{-C t}(\norm{(u_0,\partial_t u_0)}_Z + \norm{(\partial_x u_0, \partial_t \partial_x u_0)}_Z).
    \end{align*}
\end{lemma}

\begin{proof}
    Since $(u_0,\partial_tu_0)^{\top} \in D(\cA_\perp)$, so is $\tcS(t)(u_0,\partial_tu_0)^{\top}$.
    Formally differentiating equation \eqref{equation_homogeneous} with respect to space and setting $v = \partial_x u$ yields the equation 
    \begin{align*}
        \partial_t^2v + \nu \partial_t v + \mathcal{L}_0 v + \cJ u = 0,
    \end{align*}
    where the operator $\cJ$ denotes the commutator
    \begin{align*}
        \cJ = \llb \partial_x, \mathcal{L}_0 \rrb.
    \end{align*}
    The equation translates into the system 
    \begin{align*}
        \partial_t  \begin{pmatrix} v \\ w \end{pmatrix} = \cA_\perp \begin{pmatrix} v \\ w \end{pmatrix} + \begin{pmatrix}
            0 \\
            \cJ u
        \end{pmatrix},
    \end{align*}
    where, notably, the operator $\cJ :H^2_\perp \to H^1_\perp$ is bounded. Thus, via Duhamel's formula, a mild solution is given by 
    \begin{align*}
        (v,w)^{\top}(t) = \tcS(t)(v,w)^\top + \int_0^t \tcS(t-s)(0,\cJ u)^\top \: ds.
    \end{align*}
    Then using Remark \ref{remark_hypocoercivity_S_norm}, we obtain 
    \begin{align*}
        \norm{(v,w)(t)}_Z \lesssim & \norm{(v_0,w_0)}_Z e^{-C t} + \int_0^t e^{-C (t-s)}\norm{(u(s),\partial_t u(s))}_Z \: ds \\
        \lesssim & \norm{(v_0,w_0)}_Z e^{-C t} + \norm{(u_0,\partial_t u_0)}_Z te^{-C t},
    \end{align*}
    yielding the result.
\end{proof}

\begin{corollary}\label{corollary_S_H1}
    The operator $\tcS (t) : D(\cA_\perp) \to D(\cA_\perp)$ is a bounded operator with bound 
    \begin{align*}
        \big\|\tcS(t)\big\|_{D(\cA_\perp) \to D(\cA_\perp)} \lesssim (1 + t)e^{-C t}.
    \end{align*}
\end{corollary}

\begin{corollary}\label{corollary_int_welldefined}
    Let $(u,v)^{\top} \in C^0(\R;D(\cA_\perp))$ be bounded. Then 
    \[
    \int_{-\infty}^\gamma \tcS(\gamma-s)(u,v)^{\top}(s) \: ds \in D(\cA_\perp),
    \]
    is well-defined for all $\gamma \in \R$.
\end{corollary}

\begin{proof}
    Using Corollary \ref{corollary_S_H1}, we may estimate 
    \begin{align*}
        \norm{\int_{-\infty}^\gamma \tcS(\gamma-s)(u(s),v(s))^{\top} \:ds}_{D(\cA_\perp)} \lesssim \int_{-\infty}^\gamma (1+s)e^{-C (\gamma - s)}\: ds \; \sup_{r \in \R}\norm{(u,v)(r)}_{D(\cA_\perp)}.
    \end{align*}
    Since $D(\cA_\perp)$ is complete equipped with the graph norm, the claim follows.
\end{proof}

\begin{lemma}\label{lemma_periodic_int}
    Let $(u,v) \in C^0(\R;D(\cA_\perp))$ be $T$-periodic. Then so is the mapping 
    \[
    t \mapsto \int_{-\infty}^t \tcS(t-s)(u,v)^{\top}(s) \: ds.
    \]
\end{lemma}

\begin{proof}
    The proof is a simple calculation. Using the transformation $r = T + t - s$ we obtain
    \begin{align*}
        \int_{-\infty}^{T+t}\tcS(T+t - s)(u,v)^{\top}(s) \: ds = \int_0^\infty \tcS(r)(u,v)^{\top}(T + t - r) \: dr.
    \end{align*}
    Now using periodicity of $(u,v)$ and reapplying the transformation $s = t-r$, we obtain 
    \begin{align*}
        \int_0^\infty \tcS(r)(u,v)^{\top}(T + t - r) \: dr = \int_{-\infty}^t \tcS(t-s)(u,v)^{\top}(s) \: ds.
    \end{align*}
\end{proof}

We now have gathered all the tools needed to prove Theorem \ref{dynamical_linearized}.

\begin{proof}[Proof of Theorem \ref{dynamical_linearized}]
    By $H^k(\R) = H^k_\perp(\R) \oplus \text{span} \{ \partial_x \theta_0 \}$, we may split the initial data into the orthogonal parts $u_0 = u_{0,\perp} + \lambda \partial_x \theta_0$. By linearity of the equation, we may then solve the problem for each of the initial values $u_0^1 = u_{0,\perp}$ and $u_0^2 = \lambda \partial_x \theta_0$. For initial value $u_0^1$ we may reformulate the dynamical problem as the abstract inhomogeneous Cauchy-problem
    \begin{align}\label{abstract_cauchy}
        \partial_t(u,v)^{\top} = \cA_\perp(u,v)^{\top} + (0,f)^{\top},
    \end{align}
    with underlying Hilbert space $Z$ and domain $D(\cA_\perp)$.
    By Corollary \ref{corollary_int_welldefined} and Lemma \ref{lemma_periodic_int}, the unique $T$-periodic mild solution is given by $t \mapsto \int_{-\infty}^t \tcS(t-s) (0,f)^{\top}(s) \: ds$. Moreover, this solution is a classical one, since the ``initial value'' is 
    \[
    \int_{-\infty}^0 \tcS(-s)(0,f(s))^{\top} \: ds \in D(\cA_\perp). 
    \]
    The other problem is then given by the equation 
    \begin{align*}
        \ddot \lambda + \nu \dot \lambda = 0,
    \end{align*}
    since $\mathcal{L}_0 \partial_x \theta_0 = 0$ (here $\dot{\lambda} = d\lambda / dt$). $\dot\lambda$ is thus given by $\dot\lambda(t) = \dot\lambda(0)e^{-\nu t}$. The only periodic solutions are constant functions $\lambda(t) = \lambda(0)$, concluding our proof.
    \end{proof}


    \begin{proposition}\label{T_0_closed}
        The operator $\cT_0$ from Definition \ref{defspaces} is a closed operator.
        In particular, $D(\cT_0)$ equipped with the graph norm is a Banach space.
    \end{proposition}

    \begin{proof}
        We examine an arbitrary sequence $(u_k)_{k \in \N} \subset (C^2_tL^2_\perp \cap C^1_tH^1_\perp \cap C^0_t H^2_\perp)(Q_T)$ such that 
        \begin{gather*}
            u_k \to u \in (C^2_tL^2_\perp \cap C^1_tH^1_\perp \cap C^0_t H^2_\perp)(Q_T), \\
            \cT_0 u_k \to y \in C^0_tH^1_\perp(Q_T).
        \end{gather*}
        Then there exists a unique solution 
        \[
        \tilde{u} := [\int_{-\infty}^t \tcS(t-s)(0,y(s))^{\top} \: ds]_1 \in (C^2_tL^2_\perp \cap C^1_tH^1_\perp \cap C^0_t H^2_\perp)(Q_T) 
        \]
        (here $[ \cdot]_1$ denotes the first component of a vector-valued function), solving $\cT_0 \tilde{u} = y$ (see Theorem \ref{dynamical_linearized}).
        We need to show that $u = \tilde{u}$.
        Setting $y_k := \cT_0 u_k$, again by Theorem \ref{dynamical_linearized} we know that $u_k = [\int_{-\infty}^t \tcS(t-s)(0,y_k(s))^{\top} \, ds]_1$. Then 
        \begin{align}
            \sup_{t \in \R}\norm{u_k(t) - \tilde{u}(t)}_{H^1} \lesssim & \sup_{t \in \R}\norm{\int_{-\infty}^t \tcS(t-s)(0,y_k(s) - y(s))^{\top} \, ds}_Z  \label{estimate_T_closed} \\
            \lesssim & \sup_{s \in \R}\norm{y_k(s) - y(s)}_{L^2} \to 0,
        \end{align}
        for $k \to \infty$ by convergence of $y_k \to y$ in $(C^0_tH^1_\perp)(Q_T)$. Therefore, $u_k \to \tilde{u}$ in $C^0_tH^1_\perp(Q_T)$. By uniqueness of the limit, it holds that $u = \tilde{u} \in D(\cT_0)$ and $\cT_0 u = y$. This concludes our proof.
    \end{proof}

    \subsection{The nonlinear problem} We now examine the nonlinear problem. The existence result will be an application of the implicit function theorem in the appropriate (periodic) function spaces. We start by defining the $T/2$-symmetric (and hence $T$-periodic) function spaces suitable for our needs.

    \begin{definition}
        Let ${\widehat{\pi}}$ be the linear operator on the set of measurable functions in two variables given by 
        \begin{align*}
            {\widehat{\pi}}f(x,t) := -f(-x,t + T/2).
        \end{align*}
        Then we define the subspaces of $T/2$-symmetric functions $C^k_{\widehat{\pi}},D(\cT_0)_{\widehat{\pi}}$ and $C^0_tH^1_x(Q_T)_{\widehat{\pi}}$ by
        \begin{align*}
            C^2_{\widehat{\pi}}:= & \{X \in C^2(\bb{T}_T) \big | X(t) = {\widehat{\pi}}X(t)\},\\
            D(\cT_0)_{\widehat{\pi}} := & \{\chi \in D(\cT_0) \big| \chi(x,t)  = {\widehat{\pi}}\chi(x,t)\}
        \end{align*}
        and 
        \begin{align*}
            C^0_tH^1_x(Q_T)_{\widehat{\pi}} := \{f \in C^0_tH^1_x(Q_T) \big| f(x,t) = {\widehat{\pi}}\chi(x,t)\}.
        \end{align*}
    \end{definition}

    \begin{definition}\label{definition_G}
        Define the nonlinear operator $\cG : D(\cT_0)_{\widehat{\pi}}\times C^2_{\widehat{\pi}}\times \R\to C^0_tH^1_x(Q_T)$ by
        \begin{align*}
            {\cG}(\chi,X,\varep) := \partial_t^2 \theta(x,t) + \nu \partial_t \theta(x,t) + \nabla \mathcal{E}(\theta)(x,t) - \varep H(t) \cos(\theta(x,t)),
        \end{align*}
        where 
        \begin{align*}
            \theta(x,t) = \theta_0(x + X(t)) + \chi(x,t).
        \end{align*}
        If ${\cG}(\chi,X,\varep) = 0$ then $\theta$ is a solution to the reduced LLG dynamics \eqref{LLG-H}.
    \end{definition}

    \begin{lemma}\label{lemma_G_symmetry}
        The range of the operator ${\cG}$ is a subset of $C^0_tH^1_x(\R)_{\widehat{\pi}}$.
    \end{lemma}

    \begin{proof}
        As $\theta_0$ is odd we find 
        \[
        \begin{aligned}
            \theta(x,t) &= \theta_0(x + X(t)) + \chi(x,t) \\
            &=-\theta_0(-x + X(t + T/2)) -\chi(-x,t+T/2) \\
            &= -\theta(-x,t+T/2).
        \end{aligned}
        \]
        A simple calculation shows that
        \[
        \partial_t^2 {\widehat{\pi}}f + \nu \partial_t {\widehat{\pi}}f + \nabla \mathcal{E}({\widehat{\pi}}f) - \varep H(t) \cos({\widehat{\pi}}f) = {\widehat{\pi}}(\partial_t^2 f + \nu \partial_t f + \nabla \mathcal{E}(f) - \varep H(t) \cos(f)).
        \]
        Thus, ${\cG}$ satisfies ${\widehat{\pi}}{\cG} = {\cG}$, as claimed.
    \end{proof}

    \begin{lemma}\label{lemma_G_differentiable}
        The operator ${\cG}$ is continuously Fréchet-differentiable.
    \end{lemma}

    \begin{proof}
        The proof is a simple but lengthy calculation. Indeed, evaluating 
        \begin{align*}
            \frac d {ds} {\cG}(\chi + s \rho,X+sY,\varep + s\delta)
        \end{align*}
        shows that ${\cG}$ is differentiable everywhere and the linear functional
        \begin{align*}
            D{\cG}(\chi,X,\varep) \in \ccL(D(\cT_0)_{\widehat{\pi}}\times C^2_{\widehat{\pi}}\times \R,C^0_tH^1_x(Q_T)_{\widehat{\pi}})
        \end{align*} depends continuously on $(\chi,X,\varep)$ in operator norm.
    \end{proof}

    \begin{definition}
        Define the linear operator $\Gamma_0 :D(\cT_0)_{\widehat{\pi}} \times C^2_{\widehat{\pi}}$ by 
         \begin{align*}
            \Gamma_0(\rho,Y) := \frac d {ds}\bigg|_{s=0} {\cG}(\chi + s\rho, X + sY,0) =  \begin{pmatrix}
                \cT_0 & \partial_x \theta_0(\partial_t^2 + \nu \partial_t)  \end{pmatrix}\begin{pmatrix}
                    \rho \\ Y
                \end{pmatrix},
        \end{align*}
        i.e. the linearization of ${\cG}$ at $\theta_0$ for $\varep = 0$.
    \end{definition}

    \begin{lemma}\label{lemma_Gamma_bounded}
        $\Gamma_0$ is bounded.
    \end{lemma}

    \begin{proof}
        Since $\cT_0:D(\cT_0)_{\widehat{\pi}}\to C^0_tH^1_x(Q_T)_{\widehat{\pi}}$ is a bounded linear operator, we obtain 
        \begin{align*}
            \norm{\Gamma_0(\rho,Y)}_{C^0_tH^1_x} \leq\norm{\cT_0 \rho}_{C^0_tH^1_x} + \sup_{t \in [0,T]}\abs{\ddot Y(t) + \nu \dot Y(t) }\norm{\partial_x \theta_0}_{C^0_tH^1_x} \lesssim \norm{\rho}_{D(\cT_0)} + \norm{Y}_{C^2}
        \end{align*}
        $\Gamma_0$ is thus bounded.
    \end{proof}

    \begin{lemma}\label{lemma_Gamma_inv}
        $\Gamma_0$ is an isomorphism.
    \end{lemma}

    \begin{proof}
         We need to show that there exists a unique solution $(\rho,Y) \in D(\cT_0)_{\widehat{\pi}} \times C^2_{\widehat{\pi}} $ to $\Gamma_0(\rho,Y) = f$ for every $f \in C^0_tH^1_x(Q_T)_{\widehat{\pi}}$ that is,
        \begin{align*}
            \partial_x \theta_0(\ddot{Y} + \nu \dot{Y}) + \partial_t^2 \rho + \nu \partial_t \rho + \mathcal{L}_0 \rho = f.
        \end{align*}
        Assuming $f_\perp = f - \partial_x \theta_0(\ddot{Y} + \nu\dot{Y}) \perp \partial_x \theta_0$, we may apply Theorem \ref{dynamical_linearized} and obtain a unique $T$-periodic solution $\rho \in D(\cT_0)$. Since 
        \begin{align*}
            \nabla \mathcal{E}({\widehat{\pi}}u) = {\widehat{\pi}} \nabla \mathcal{E}(u),
        \end{align*}
        we obtain 
        \[
        \begin{aligned}
            \mathcal{L}_0{\widehat{\pi}}u = \frac d {ds} \bigg|_{s = 0} \nabla \mathcal{E}(\theta_0 + s{\widehat{\pi}}u) &=\frac d {ds} \bigg|_{s = 0} \nabla \mathcal{E}({\widehat{\pi}}(\theta_0 + u)) \\
            &= \frac d {ds} \bigg|_{s = 0}{\widehat{\pi}} \nabla \mathcal{E}(\theta_0 + su) \\
            &= {\widehat{\pi}}\mathcal{L}_0 u.
        \end{aligned}
        \]
        Thus,
        \begin{align*}
            \cT_0{\widehat{\pi}}\rho = {\widehat{\pi}}\cT_0\rho = {\widehat{\pi}}f = f.
        \end{align*}
        ${\widehat{\pi}}\rho$ is then also the unique $T$-periodic solution, so ${\widehat{\pi}}\rho = \rho$. The problem thus reduces to finding a unique $Y \in C^2_{\widehat{\pi}}$ satisfying $\ddot{Y} + \nu \dot{Y} = \skalprod{f(t)}{\partial_x \theta_0}/\norm{\partial_x \theta_0}_{L^2}^2$, the projection of $f(t)$ onto the subspace spanned by $\partial_x \theta_0$. Take any $f\in C^0_tH^1_x(Q_T)_{\widehat{\pi}}$. Then 
        \begin{align*}
            \int_0^T \skalprod{f(\cdot,t)}{\partial_x \theta_0} \: dt = \int_{T/2}^T \skalprod{f(\cdot,t)}{\partial_x \theta_0} -\skalprod{f(-\cdot,t)}{\partial_x \theta_0} \: dt = 0,
        \end{align*}
        since $\partial_x \theta_0$ is even. Set $g(t) := \skalprod{f(t)}{\partial_x \theta_0}/\norm{\partial_x \theta_0}_{L^2}^2$ and note that $g(t) = -g(t + T/2)$. We obtain a solution to
        \begin{align*}
            \dot{W} + \nu W = g(t),
        \end{align*}
        which, by variation of constants, reads
        \begin{align*}
            W(t) = \int_{-\infty}^t e^{-\nu (t-s)}\tilde{g}(s) \:ds.
        \end{align*}
        Periodicity and convergence of the integral follow as in the proof of Theorem \ref{dynamical_linearized}, since $g$ is $T$-periodic and continuous. Uniqueness follows by solving the homogeneous equation, whose solution is $W_h(t) = Ce^{-\nu t}$, and requiring it to be $T$-periodic. Moreover, this solution satisfies $W(t) = -W(t+T/2)$. The $T$-periodic solutions are thus given by $Y_c(t) = C + \int_0^t W(s) \: ds$. If we now choose $Y(T/2) = 0$, we obtain 
        \begin{align*}
            Y(t+T/2) = \int_0^{T/2+t}W(s) \: ds = Y(T/2) + \int_{T/2}^{T/2 +t} W(s) \: ds = -Y(t).
        \end{align*}
        $\Gamma_0$ is thus invertible.
    \end{proof}

    \subsection{Proof of Theorem \ref{theorem_existence}}
%
        Since the linearization $\Gamma_0 = D_{\chi,X}{\cG}(0,0,0)$ is a bounded isomorphism between Banach spaces, by the implicit function theorem there exist a neighborhood $(-\varep_0,\varep_0)$ of $0 \in \R$, a neighborhood $V \subset D(\cT_0)_{\widehat{\pi}}\times C^2_{\widehat{\pi}}$ of $(0,0)$ and a continuously differentiable function $g:(-\varep_0,\varep_0) \to V$, such that ${\cG}(\chi,X,\varep) = 0$ if and only if $g(\varep) = (\chi,X)$. Since $g$ is differentiable, the norm estimates $\norm{\chi}_{D(\cT_0)} = O(\varep)$, $\norm{X}_{C^2} = O(\varep)$ follow. Set $\frac{d}{d\varep}\big|_{\varep = 0}g(\varep) = (\rho,Y)$ and test $\frac d {d\varep}\big|_{\varep = 0}{\cG}(g(\varep),\varep)$ with $\partial_x \theta_0$. We obtain
        \begin{align*}
            0 = \norm{\partial_x \theta_0}_{L^2}^2(\ddot Y + \nu \dot Y) - H(t)\skalprod{\partial_x \theta_0}{\cos(\theta_0)}_{L^2}
        \end{align*}
        and thus 
        \begin{align*}
            \ddot Y + \nu \dot Y = 2\norm{\partial_x \theta_0}_{L^2}^{-2}H(t).
        \end{align*}
        This concludes the proof. \qed

\begin{remark}
\label{varepzerostatic}
It is to be observed that, since ${\cG}(0,0,0) = 0$ and by uniqueness of the solutions (up to translations), when $\varepsilon \to 0^+$ we recover the static N\'eel wall, $\brt_{|\vep=0}(\cdot,t) = \theta_0(\cdot)$, which is time periodic everywhere with any period.
\end{remark}
%

\section{Linearization and generation of the evolution system}
\label{secspectral}

In this Section we linearize the Landau-Lifshitz-Gilbert reduced model around the time periodic solution from Theorem \ref{theorem_existence} and prove the generation of a propagator or evolution system

\subsection{Linearized equation for perturbations}
\label{seclinedp}

Let $\brt^\vep = \brt^\vep(x,t)$ be the time-periodic N\'eel wall solution from Theorem \ref{theorem_existence} for a fixed $\vep \in (- \vep_0, \vep_0)$ sufficiently small. By construction, the mapping
\[
(- \vep_0, \vep_0) \ni \vep \mapsto \brt^\vep \in (C_t^2L_x^2 \cap C_t^1 H_x^1 \cap C^0_t H_x^2) (Q_T),
\]
is of class $C^1$ and the solution for $\vep = 0$ coincides with the static N\'eel wall,
\[
\brt^0(x,t) = \brt^\vep \big|_{\vep = 0}(x,t) \equiv \theta_0(x),
\]
which is time-periodic for any period $T > 0$. For the sake of simplicity, in the sequel we suppress the $\vep$-notation and write $\brt := \brt(x,t) = \brt^\vep(x,t)$ to denote the time-periodic N\'eel wall for any fixed $\vep \in (- \vep_0, \vep_0)$. The latter is a solution to the nonlinear equation,
\begin{equation}
\label{fulleq}
\partial_t^2 \theta + \nu \partial_t \theta + \nabla \cE (\theta) = \vep H(t) \cos \theta.
\end{equation}

Let us consider a solution to \eqref{fulleq} of the form $\brt(x,t) + u(x,t)$, where now $u$ denotes a small perturbation of the time-periodic N\'eel wall. Upon substitution into \eqref{fulleq} and since $\brt$ itself solves \eqref{fulleq}, we obtain the following nonlinear equation for the perturbation,
\begin{equation}
\label{nonlinu}
\partial_t^2 u + \nu \partial_t u + \nabla \cE (\brt + u) - \nabla \cE (\brt + u) - \vep H(t) ( \cos (\brt + u) - \cos \brt ) = 0.
\end{equation}
Making the expansions around $\brt$,
\[
\begin{aligned}
\nabla \cE(\brt + u) - \nabla \cE(\brt) &= D^2 \cE (\brt) u + O(u^2),\\
\cos (\brt + u) - \cos \brt &= - (\sin \brt) u + O(u^2),
\end{aligned}
\]
and omitting the higher order terms, we arrive at the linearized equation (around $\brt$) for the perturbation,
\begin{equation}
\label{linequ}
\partial_t^2 u + \nu \partial_t u + D^2 \cE (\brt) u + \vep H(t) ( \sin \brt ) u = 0.
\end{equation}
Let us examine the term $D^2 \cE (\brt)$. From
\[
\begin{aligned}
\frac{d}{ds} \cE (\brt + s u) \big|_{s=0} &= \int_\R \big( \partial_x \brt \partial_x u - \cos \brt \sin \brt \, u - \sin \brt u (- \Delta)^{1/2} \cos \brt \big) \, dx, \quad \text{and}\\
\frac{d}{ds} \nabla \cE (\brt + s u) \big|_{s=0} &= - \partial_x^2 u + u \cos \brt (1 + (-\Delta)^{1/2}) \cos \brt - \sin \brt (1 + (-\Delta)^{1/2}) ((\sin \brt) u),
\end{aligned}
\]
we recognize that the expression
\[
\begin{aligned}
D^2 \cE(\brt) u &= - \partial_x^2 u + s_{\brt}(x,t) (1 + (-\Delta)^{1/2}) ( s_{\brt}(x,t) u) - c_{\brt}(x,t) u \\ &=: \cL_{\brt} u,
\end{aligned}
\]
coincides with the linearization around the solution $\brt(x,t)$ (compare to the corresponding expression for $\cL_0$), and where 
\[
s_{\brt}(x,t) := \sin \brt(x,t), \qquad c_{\brt}(x,t) := \cos \brt(x,t) (1 + (-\Delta)^{1/2}) \cos \brt(x,t),
\]
are now coefficients depending on $x$ and $t$, and time-periodic. 

In other words, the linearization around the time-periodic N\'eel walls, since it involves coefficients which are time-dependent, now defines a family of linear operators parametrized by $t \geq 0$. Indeed, for each fixed $t \in [0,T]$, the operator
\begin{equation}
\label{defLbt}
\left\{
\begin{aligned}
\cL_{\brt} (t) &: L^2 \to L^2,\\
D(\cL_{\brt}(t)) &= H^2,\\
\cL_{\brt}(t) u &:= - \partial^2_x u + \mathcal{S}_{\brt}(t) u - c_{\brt} u, \qquad u \in D(\cL_{\brt}(t)) = H^2,
\end{aligned}
\right.
\end{equation}
is a closed, linear, densely defined operator (the domain of the family is fixed, $D(\cL_{\brt}(t)) = H^2$, for all $t$) where, again for each $t$, the non-local, linear operator $\mathcal{S}_{\brt}(t)$ is defined as
\begin{equation}
\label{defS0}
\left\{
\begin{aligned}
\mathcal{S}_{\brt}(t) &: L^2 \to L^2,\\
D(\mathcal{S}_{\brt}(t)) &= H^1,\\
\mathcal{S}_{\brt}(t) u &:= s_{\brt}(x,t) (1+(-\Delta)^{1/2}) (u s_{\brt}(x,t)), \qquad u \in D(\mathcal{S}_{\brt}(t)) = H^1.
\end{aligned}
\right.
\end{equation}

Use the precise form of the time-periodic N\'eel wall (see Theorem \ref{theorem_existence}),
\[
\brt(x,t) = \theta_0 (x + X(t)) + \chi(x,t),
\]
and make the Taylor expansions 
\[
\cos (\theta_0 + \chi) = \cos \theta_0 - \chi \sin \theta_0 + O(\chi^2), \quad \sin (\theta_0 + \chi) = \sin \theta_0 + \chi \cos \theta_0 + O(\chi^2),
\]
to arrive at
\begin{equation}
\label{decompL}
\cL_{\brt} = \cL_0 + \cL^{(1)}_\chi + \cL^{(2)}_\chi,
\end{equation}
where
\[
\cL_0 u = - \partial_x^2 u + \sin \theta_0 (1 + (-\Delta)^{1/2}) ((\sin \theta_0) u) + [\cos \theta_0 (1 + (-\Delta)^{1/2}) \cos \theta_0] u,
\]
is the non-local linearized operator around the static N\'eel wall defined in \eqref{defL0} and where,
\begin{equation}
\label{exprL1chi}
\begin{aligned}
\cL^{(1)}_\chi u &:= \sin \theta_0 (1 + (-\Delta)^{1/2}) ((\chi \cos \theta_0) u) + \chi \cos \theta_0 (1 + (-\Delta)^{1/2}) (( \sin \theta_0)u) \\
& \; - [\cos \theta_0 (1 + (-\Delta)^{1/2}) \chi \sin \theta_0] u - [\chi \sin \theta_0 (1 + (-\Delta)^{1/2}) \cos \theta_0] u,
\end{aligned}
\end{equation}
and
\begin{equation}
\label{exprL2chi}
\begin{aligned}
\cL^{(2)}_\chi u &:= \chi \cos \theta_0 (1 + (-\Delta)^{1/2}) ((\chi \cos \theta_0) u) - [\chi \sin \theta_0 (1 + (-\Delta)^{1/2}) (\chi \sin \theta_0)] u \\
& \; + \sin \theta_0 (1 + (-\Delta)^{1/2}) (O(\chi^2) u) + \chi \cos \theta_0 (1 + (-\Delta)^{1/2}) (O(\chi^2) u) \\
& \; + [ \cos \theta_0 (1 + (-\Delta)^{1/2}) O(\chi^2)] u + [O(\chi^2) (1 + (-\Delta)^{1/2}) \cos \theta_0]u\\
& \; - [\chi \sin \theta_0 (1 + (-\Delta)^{1/2}) O(\chi^2)]u - [O(\chi^2) (1 + (-\Delta)^{1/2}) (\chi \sin \theta_0)]u\\
& \; + [O(\chi^2) (1 + (-\Delta)^{1/2}) O(\chi^2)] u,
\end{aligned}
\end{equation}
are non-local, linear operators with time-periodic coefficients depending on $(x,t)$.

Notice that $\cL^{(1)}_\chi = O(|\chi|)$ and $\cL^{(2)}_\chi = O(\chi^2)$, and therefore the operator $\cL_{\brt}$ is a perturbation of the static operator $\cL_0$ of order $O(|\chi|)$. This observation is made precise in the following result.

\begin{lemma}
\label{lemperL}
There exists a uniform constant $C > 0$ such that
\begin{equation}
\label{estpertu}
\| (\cL_{\brt} - \cL_0) u \|_{L^2} \leq C \| \chi \|_{C_t^0H_x^2} \| u \|_{H^1},
\end{equation}
for all $u \in H^1$ and all $t \geq 0$.
\end{lemma}
\begin{proof}
First, it known that the linear operator $1 + (-\Delta)^{1/2}$ is bounded from $H^1$ to $L^2$ (see, e.g., Proposition 2.2 in Capella \emph{et al.} \cite{CMMP25}). Hence, from the continuous embedding $H^1 \hookrightarrow L^\infty$ and from the expression \eqref{exprL1chi} for $\cL_\chi^{(1)}$ we obtain
\[
\begin{aligned}
\| \cL^{(1)}_\chi u \|_{L^2} &= \big\| \sin \theta_0 (1 + (-\Delta)^{1/2}) ((\chi \cos \theta_0) u) + \chi \cos \theta_0 (1 + (-\Delta)^{1/2}) (( \sin \theta_0)u) \\ & \;\; - [\cos \theta_0 (1 + (-\Delta)^{1/2}) \chi \sin \theta_0] u - [\chi \sin \theta_0 (1 + (-\Delta)^{1/2}) \cos \theta_0] u \big\|_{L^2} \\
&\leq C \| \chi \|_{C_t^0H_x^2} \| u \|_{H^1},
\end{aligned}
\]
for some uniform $C > 0$ which may depend on $T > 0$. Likewise, from expression \eqref{exprL2chi} and from a similar argument we arrive at
\[
\| \cL^{(2)}_\chi u \|_{L^2} \leq C \| \chi \|_{C_t^0H_x^2} (\|u \|_{L^2} + \| u \|_{H^1}) \lesssim \| \chi \|_{C_t^0H_x^2} \| u \|_{H^1}.
\]
This yields the result.
\end{proof}

Now, let us set $v := \partial_t u$ in order to recast the linear equation \eqref{linequ} for the perturbation as a system of the form
\begin{equation}
\label{systlin}
\partial_t \begin{pmatrix} u \\ v \end{pmatrix} = \begin{pmatrix} 0 & I \\ - \cL_{\brt} - \vep H(t) \sin \brt & - \nu I \end{pmatrix}\begin{pmatrix} u \\ v \end{pmatrix} =: \cA_{\brt}(t) \begin{pmatrix} u \\ v\end{pmatrix},
\end{equation}
where $\cA_{\brt}(t)$ is a family (parametrized by $t \geq 0$) of time-periodic linear operators with domain 
\[
D(\cA_{\brt}(t)) = H^2 \times H^1, \qquad \text{for all} \; t \geq 0.
\]
Notice that the domain of the family is independent of $t$. Moreover, $D(\cA_{\brt}(t)) = D(\cA_0) = H^2 \times H^1$, for all $t$. Thus, for shortness let us denote
\[
D_0 := D(\cA_0) = D(\cA_{\brt}(t)) = H^2 \times H^1, \qquad \forall \; t \geq 0,
\]
as well as,
\[
\cA(t) := \cA_{\brt}(t), \qquad \forall \; t \geq 0.
\]

Hence, since the linearized operators depend on $t$, we arrive at a linear evolution problem of the form $\partial_t U = \cA(t) U$ for $U = (u, v) \in H^1 \times L^2$.

\subsection{The evolution system}

In the Hilbert space $H^1 \times L^2$, for every fixed $t \geq 0$, $\cA(t) : D_0 \subset H^1 \times L^2 \to H^1 \times L^2$ is a linear operator. Thus, the family $\{ \cA(t) \}_{t \geq 0}$ determines an evolution problem of the form
\begin{equation}
\label{evolsyst}
\left\{
\begin{aligned}
\frac{dU}{dt} &= \cA(t) U,\\
U(s) &= U_s, & 0 \leq s < t,
\end{aligned}
\right.
\end{equation}
and we are interested in the propagator operator
\begin{equation}
\label{propag}
\cU(t,s) U_s := U(t), \qquad \text{for } \, 0 \leq s < t,
\end{equation}
which denotes the solution to the evolution problem \eqref{evolsyst} for each $0 \leq s < t$, with initial datum $U_s$ at $t = s$. Recall that an evolution system, $\cU(t,s)$, $0 \leq s \leq t$, is a biparametric family of bounded operators $\{ \cU(t,s) \}_{0 \leq s \leq t} \subset \ccB(H^1 \times L^2)$ such that $\cU(s,s) = I$, $\cU(t,r) \cU(r,s) = \cU(t,s)$ for all $0 \leq s \leq r \leq t$, and the mapping $(t,s) \mapsto\cU(t,s)$ is strongly continuous for all $0 \leq s \leq t$. In order to prove the existence of the propagator we need to verify the following conditions (cf. Pazy \cite{Pa83}, \S 5.5.3):

\begin{itemize}
\item[(H$_1$)] \phantomsection
\label{H1} $\{\cA(t)\}_{t \geq 0}$ is a stable family of generators with uniform stability constants $M \geq 1$ and $\omega \in \R$.
\item[(H$_2^+$)] \phantomsection
\label{H2p} There exists a family of isomorphisms $\cQ(t)$ of $Y := D(\cA(t)) \subset X$ onto $X$ such that $\cQ(t) V$ is continuously differentiable in $t$ for all $V \in Y$ and $\cQ(t) \cA(t) \cQ(t)^{-1} = \cA(t) + \cB(t)$, where $\{\cB(t)\}_{t \geq 0}$ is a strongly continuous family of bounded operators on $X$.
\item[(H$_3$)] \phantomsection
\label{H3} $\cA(t)$ is bounded from $Y$ to $X$ for all $t$ and $t \mapsto \cA(t)$ is a continuous operator in the operator norm of $\ccB(Y,X)$, $\| \cdot \|_{Y \to X}$.
\end{itemize}

\begin{remark}
The set of conditions \hyperref[H1]{\rm{(H$_1$)}}, \hyperref[H2p]{\rm{(H$_2^+$)}} and \hyperref[H3]{\rm{(H$_3$)}} are usually referred to as the ``hyperbolic'' case (in contrast with the ``parabolic'' case in which $\cA(t)$ generates an analytic semigroup; see Pazy \cite{Pa83}, \S 5.5.3).
\end{remark}

A corollary from Lemma \ref{lemperL} is the fact that $\cA(t)$ is a relatively bounded perturbation of the block operator $\cA_0$ defined in \eqref{defA0}.

\begin{corollary}
\label{corperA}
The family of operators,
\begin{equation}
\label{defcB}
\cB_\vep(t) := \cA(t) - \cA_0, \qquad t \geq 0,
\end{equation}
is uniformly bounded in $H^1 \times L^2$.
\end{corollary}
\begin{proof}
It is clear that for each fixed value of $\vep \in (-\vep_0, \vep_0)$ and for all $t \geq 0$ we have,
\[
\left\{
\begin{aligned}
\cB_\vep(t) &= \cA(t) - \cA_0 = \begin{pmatrix} 0 & 0 \\ \widetilde{\cB}_\vep(t) & 0 \end{pmatrix},\\
\cB_\vep(t) &: H^1 \times L^2 \to H^1 \times L^2,
\end{aligned}
\right.
\]
where
\[
\widetilde{\cB}_\vep(t) := \cL_0 - \cL_{\brt} - \vep H(t) \sin \brt,
\]
with a constant domain $D(\cB_\vep(t)) \equiv H^1 \times L^2$ (notice that the Laplacian cancels out in the expression for $\cL_0 - \cL_{\brt}$ and $\cB_\vep(t)$ is defined everywhere in $H^1 \times L^2$). Therefore, for all $U = (u, v)^{\top} \in H^1 \times L^2$ there holds
\[
\begin{aligned}
\| \cB_\vep(t) U \|_{H^1 \times L^2} = \| \widetilde{\cB}_\vep(t) u \|_{L^2} &= \| (\cL_0 - \cL_{\brt}) u \|_{L^2} + \| \vep H(t) \sin \brt u \|_{L^2}\\
&\leq C ( \| \chi \|_{C_t^0H_x^2} + |\vep|) \| u \|_{H^1}\\
&\lesssim |\vep| \| u \|_{H^1},
\end{aligned}
\]
inasmuch as $\| \chi \|_{C_t^0H_x^2} = O(|\vep|)$. This implies that there exists a uniform constant $C > 0$ such that
\[
\| \cA(t) - \cA_0 \| \leq C |\vep|, \qquad \forall \, t \geq 0,
\]
and the family is uniformly bounded, $\cA(t) - \cA_0 \in \ccB(H^1 \times L^2)$ for all $t \geq 0$.
\end{proof}

\begin{lemma}
\label{lemstableA}
The family of operators $\{ \cA(t) \}_{t \geq 0}$ is a stable family of infinitesimal generators with stability constants $M \equiv 1$ and $\omega + C|\vep| \in \R$ for some uniform $C > 0$.
\end{lemma}
\begin{proof}
From Proposition \ref{propA0} (c), we know that $\cA_0 : D_0 \subset H^1 \times L^2 \to H^1 \times L^2$ is the infinitesimal generator of a $C_0$-semigroup, $\{ e^{t\cA_0} \}_{t \geq 0}$, of quasi-contractions in $H^1 \times L^2$ and that there exists $\omega \in \R$ such that
\[
\left\| e^{t \cA_0} U \right\|_{H^1 \times L^2} \leq e^{\omega t} \| U \|_{H^1 \times L^2},
\]
for all $U \in H^1 \times L^2$ and all $t \geq 0$. Therefore, it constitutes a stable (actually, constant) family of infinitesimal generators with stability constants $M \equiv 1$ and $\omega \in \R$. From Corollary \ref{corperA} we know that $\cB_\vep(t) = \cA(t) - \cA_0$ satisfies $\| \cB_\vep(t) \| \leq C |\vep|$ for some uniform $C > 0$ and all $t$. Hence, we invoke Theorem 5.2.3 in Pazy \cite{Pa83} to conclude that $\cA_0 + \cB_\vep(t)$ is a stable family of infinitesimal generators with stability constants $M = 1$ and $\omega + C|\vep| \in \R$. This implies, in particular, that
\begin{equation}
\label{resAt}
\rho(\cA(t)) \supset (\omega + C|\vep|, \infty), \quad \forall \, t \geq 0,
\end{equation}
and that,
\[
\Big\| \prod_{j=1}^k \cR(\lambda; \cA(t_j)) \Big\| \leq (\lambda - \omega - C|\vep|)^{-k},
\]
for $\lambda \in \R$, $\lambda > \omega + C|\vep| $ and for every finite (ordered) sequence $0 \leq t_1 \leq \ldots \leq t_k$, $k \in \N$ (see Definition 5.2.1 in Pazy \cite{Pa83}).
\end{proof}

We are now ready to prove the existence of the evolution system (or propagator) $\cU$.

\begin{theorem}
\label{thmexistpropag}
There exists a unique evolution system $\cU = \cU(t,s)$ in $H^1 \times L^2$, for all $0 \leq s \leq t \leq T$, satisfying:
\begin{itemize}
\item[\rm{(E$_1$)}] \phantomsection
\label{E1} $\| \cU(t,s) \| \leq e^{(\omega + C|\vep|)(s-t)}$ for all $0 \leq s \leq t$.
\item[\rm{(E$_2$)}] \phantomsection
\label{E2} There holds
\[
\frac{\partial^+}{\partial t} \cU(t,s) V \Big|_{t = s} = \cA(s) V,
\]
for all $V \in D_0 = H^2 \times H^1$ and a.e. on $\, 0 \leq s \leq t$.
\item[\rm{(E$_3$)}] \phantomsection
\label{E3} There holds 
\[
\frac{\partial}{\partial s} \cU(t,s) V = - \cU(t,s) \cA(s) V,
\]
for all $V \in D_0$ and a.e. on $\, 0 \leq s \leq t$.
\item[\rm{(E$_4$)}] \phantomsection
\label{E4} $\cU(t,s) D_0 \subset D_0$ for all $0 \leq s \leq t$.
\item[\rm{(E$_5$)}] \phantomsection
\label{E5} For $V \in D_0$, $\cU(t,s) V$ is continuous in $D_0$ for $0 \leq s \leq t$.
\end{itemize}
Moreover, for every $V \in D_0 = H^2 \times H^1$, $\cU(t,s) V$ is the unique $D_0$-valued solution to the initial value problem
\begin{equation}
\label{invalprob}
\left\{
\begin{aligned}
\frac{dU}{dt} &= \cA(t) U, & \quad 0 \leq s \leq t,\\
U(s) &= V.
\end{aligned}
\right.
\end{equation}
\end{theorem}
\begin{proof}
From Theorem 5.4.6 in Pazy \cite{Pa83}, it then suffices to verify conditions \hyperref[H1]{\rm{(H$_1$)}}, \hyperref[H2p]{\rm{(H$_2^+$)}} and \hyperref[H3]{\rm{(H$_3$)}} for the family of operators $\cA(t)$.

From Lemma \ref{lemstableA}, we already know that condition \hyperref[H1]{\rm{(H$_1$)}} holds. In order to verify \hyperref[H2p]{\rm{(H$_2^+$)}}, let us observe that, since $1 \in \rho(\cA_0)$ (see Proposition \ref{propA0} (b)), we can write
\[
(I - \cA_0) \cA(t) (I - \cA_0)^{-1} = \cA(t) + \llb I- \cA_0, \cA(t) \rrb (I - \cA_0)^{-1}.
\]
Let us examine the family of commutators, $\llb I- \cA_0, \cA(t) \rrb$. From the expressions for $\cA_0$ and $\cA(t)$ we clearly obtain
\[
\begin{aligned}
\llb I- \cA_0, \cA(t) \rrb  &= (I - \cA_0) \cA(t) - \cA(t) (I - \cA_0)\\
&= \begin{pmatrix} \cL_{\brt} - \cL_0 + \vep H(t) \sin \brt & 0 \\ - \nu (\cL_{\brt} - \cL_0) - \nu \vep H(t) \sin \brt & \cL_{\brt} - \cL_0 + \vep H(t) \sin \brt \end{pmatrix} \\
&= \begin{pmatrix} - \widetilde{\cB}_\vep(t) & 0 \\ - \nu (\cL_{\brt} - \cL_0) - \nu \vep H(t) \sin \brt  & -\widetilde{\cB}_\vep(t)\end{pmatrix}.
\end{aligned}
\]

Thanks to Lemma \ref{lemperL} and Corollary \ref{corperA} it is clear that $\llb I- \cA_0, \cA(t) \rrb$ is a family of uniformly bounded operators satisfying
\[
\| \, \llb I- \cA_0, \cA(t)\rrb \, \| \leq C |\vep|,
\]
for some uniform $C > 0$ and all $t \geq 0$. Since $(I - \cA_0)^{-1} : D_0 \to H^1 \times L^2$ is bounded and $\llb I- \cA_0, \cA(t) \rrb$ is uniformly bounded for all $t$, we conclude that $\cQ(t) := (I- \cA_0) : D_0 \to H^1 \times L^2$ is clearly continuously differentiable in $t$ (actually, constant) and constitutes a family of isomorphisms for which 
\[
\llb I- \cA_0, \cA(t) \rrb (I - \cA_0)^{-1} : D_0 \to H^1 \times L^2,
\]
is a family of strongly continuous operators on $H^1 \times L^2$. This verifies condition \hyperref[H2p]{\rm{(H$_2^+$)}}.

Finally, condition \hyperref[H3]{\rm{(H$_3$)}} is evidently satisfied because the domains are constant in $t$, $D(\cA(t)) = D_0 = H^2 \times H^1$, and $\cA(t)$ is a bounded operator from $D_0$ to $H^1 \times L^2$. This finishes the proof.
\end{proof}

\subsection{Translation invariance}

In this section we show that, thanks to the translation invariance of the time-periodic N\'eel wall, its space derivative is an invariant direction for the propagator. To that end, first we need to verify that, as a consequence of the method of construction of the solution, the latter carries higher regularity. 

\begin{lemma}
\label{lemhigherreg}
The constructed time-periodic solution $\brt = \brt(x,t)$ satisfies:
\begin{itemize}
\item[\rm{(a)}] $\brt \in (C^2_t H^1_x \cap C^1_t H^2_x \cap C^0_t H^3_x) (Q_T)$.
\item[\rm{(b)}] For all fixed $t \in [0,T]$ there holds $\brt(\cdot, t) \in H^3(\R)$.
\end{itemize}
\end{lemma}
\begin{proof}
It follows from a bootstrapping argument on the profile equation. Indeed, by construction we know that $\brt \in (C^2_tL^2_x\cap C^1_tH^1_x \cap C^0_tH^2_x) (Q_T)$ is a solution to 
\begin{equation}
\label{star}
\partial_t^2 \brt + \nu \partial_t \brt + \nabla \cE(\brt)  = \vep H(t) \cos \brt.
\end{equation}

Upon differentiation of \eqref{star} and noticing that $\partial_t \partial_x \brt \in C_t^0 L_x^2$, $D^2\cE(\brt) \partial_x \brt \in C_t^1 L_x^2$ and $\vep H(t) (\sin \brt) \partial_x \brt \in C_t^1 L_x^2$, we obtain
\[
\partial_t^2 \partial_x \brt = - \nu \partial_t \partial_x \brt - D^2 \cE(\brt) \partial_x \brt - \vep H(t) (\sin \brt) \partial_x \brt,
\]
yielding $\brt \in C_t^2 H_x^1(Q_T)$. Moreover, since $D^2(\cE)(\brt) = \cL_{\brt}$ we notice that $\partial_x \brt$ is actually a solution of the linearized equation \eqref{linequ},
\begin{equation}
\label{dstar}
\partial_t^2 \partial_x \brt + \nu \partial_t \partial_x \brt + \cL_{\brt} \partial_x \brt + \vep H(t) (\sin \brt) \partial_x \brt = 0.
\end{equation}
If we substitute $\cL_{\brt} = - \partial_x^2 + \cS_{\brt} + c_{\brt} I$ into \eqref{dstar} we obtain
\[
\partial_x^3 \brt = \partial_t^2 \partial_x \brt + \nu \partial_t \partial_x \brt + \cS_{\brt} \partial_x \brt + c_{\brt} \partial_x \brt + \vep H(t) (\sin \brt) \partial_x \brt.
\]
Note that the right hand side of last equation belongs to $C_t^0 L_x^2$, inasmuch as $\partial_t^2 \partial_x \brt \in C_t^0 L_x^2$ and $\partial_t \partial_x \brt \in C_t^1 L_x^2$ (because $\brt \in C_t^2 H_x^1$), $\cS_{\brt} \partial_x \brt \in C_t^0 L_x^2$ (because $\cS_{\brt} : L^2 \to L^2$), $\partial_x \brt \in C_t^0 H_x^1$ and $\vep H(t) (\sin \brt) \partial_x \brt \in C_t^0 H_x^1$.Therefore, $\partial_x^3 \brt \in C_t^0 L_x^2$ and we conclude that $\brt \in C_t^0 H_x^3(Q_T)$ (which, in turn, implies (b)). The proof that $\brt \in C_t^1 H_x^2 (Q_T)$ is analogous and we omit it.
\end{proof}

As a by-product of the last result we have the following observation.
\begin{corollary}
\label{corregT}
Let us define
\begin{equation}
\label{defThetab}
\overline{\Theta} (x,t) := \big( \partial_x \brt (x,t), \, \partial_t \partial_x \brt(x,t) \big),
\end{equation}
for $x \in \R$, $t \in [0,T]$. Then there holds
\[
\overline{\Theta} \in \big( C_t^2 L_x^2 \cap C_t^1 H_x^1 \cap C_t^0 H_x^2\big)(Q_T) \times \big( C_t^1 L_x^2 \cap C_t^0 H_x^1\big)(Q_T),
\]
and for each fixed $t \in [0,T]$ we have
\[
\overline{\Theta}(\cdot, t) \in H^2(\R) \times H^1(\R).
\]
\end{corollary}

\begin{lemma}[translation invariance]
\label{lemtransinv}
$\overline{\Theta}$ is a solution to $\partial_t \overline{\Theta} = \cA(t) \overline{\Theta}$, $t \in [0,T]$.
\end{lemma}
\begin{proof}
The result is a direct consequence of $\brt$ being a solution to the linearized equation \eqref{linequ}. Indeed, from Corollary \ref{corregT} we know that $\overline{\Theta}(\cdot,t) \in D_0 = D(\cA(t))$ for all fixed $t \in [0,T]$. Since $\brt$ solves equation \eqref{dstar} we readily obtain
\[
\begin{aligned}
\partial_t \overline{\Theta} = \partial_t \begin{pmatrix} \partial_x \brt \\ \partial_t \partial_x \brt  \end{pmatrix} &= \begin{pmatrix} \partial_t \partial_x \brt  \\ -\cL_{\brt} \partial_x \brt - \vep H(t) (\sin \brt) \partial_x \brt - \nu \partial_t \partial_x \brt \end{pmatrix} \\
&= \begin{pmatrix} 0 & I \\ - \cL_{\brt} - \vep H(t) (\sin \brt) & - \nu I \end{pmatrix} \begin{pmatrix} \partial_x \brt \\ \partial_t \partial_x \brt \end{pmatrix} = \cA(t) \overline{\Theta}.
\end{aligned}
\]
\end{proof}

\begin{corollary}
\label{corpropsol}
$\cU(T,0) \overline{\Theta}(\cdot,0) = \overline{\Theta}(\cdot,0) \in H^2 \times H^1$.
\end{corollary}
\begin{proof}
From basic properties of the propagator and from translation invariance we have
\[
\begin{aligned}
\frac{\partial}{\partial s} \big( \cU(t,s) \overline{\Theta}(\cdot, s) \big) &= - \cU(t,s) \cA(s) \overline{\Theta}(\cdot, s) + \cU(t,s) \partial_s \overline{\Theta}(\cdot, s)\\
&= - \cU(t,s) \cA(s) \overline{\Theta}(\cdot, s) + \cU(t,s) \cA(s) \overline{\Theta}(\cdot, s) \equiv 0,
\end{aligned}
\]
for all $0 \leq s \leq t \leq T$. Hence, $\cU(t,s) \overline{\Theta}(\cdot, s) = \cU(t,t) \overline{\Theta}(\cdot, t) = \overline{\Theta}(\cdot, t)$ for all $0 \leq s \leq t \leq T$. Evaluating in $t = T$ and $s = 0$ we obtain the result.
\end{proof}

\section{Spectral stability}
\label{secspecstab}

In this Section we prove that the time-periodic N\'eel wall is spectrally stable. Because of the time-periodicity of the linear operators $\cA(t)$, the appropriate notion of spectral stability is given in terms of the \emph{Floquet spectrum}.

\subsection{The Floquet spectrum}

First we define the monodromy map as te propagator operator after exactly one fundamental period.

\begin{definition}
\label{defmonodspec}
Let us define the \emph{monodromy map} of the evolution system from Theorem \ref{thmexistpropag} as
\begin{equation}
\label{defmonodromy}
\cM_\vep := \cU(T,0).
\end{equation}
(We added the subscript notation in order to remind the reader of the dependence on $\vep > 0$.) The \emph{Floquet spectrum} is defined as
\begin{equation}
\label{defFloquet}
\Sigma_F := \{ \lambda \in \C \, : \, e^{T\lambda} \in \sigma(\cM_\vep) \},
\end{equation}
where $\sigma(\cM_\vep)$ denotes the spectrum of the monodromy map $\cM_\vep$.
\end{definition}

\begin{remark}
Notice that since $\cU(t,s) \in \ccB(H^1 \times L^2)$ for all $0 \leq s \leq t$, then $\cM_\vep \in \ccB(H^1 \times L^2)$ and its spectrum is computed with respect to the space $H^1 \times L^2$. Hence, $\lambda \in \Sigma_F$ if and only if $\mu = e^{\lambda T} \in \sigma(\cM_\vep)_{|{H^1 \times L^2}}$. 
\end{remark}

\begin{definition}
\label{defspectralstab}
We say the the time-periodic N\'eel wall, $\brt = \brt(x,t)$, is \emph{spectrally stable} if
\[
\Sigma_F \subset \{ \lambda \in \C \, : \, \Re \lambda \leq 0 \}.
\]
\end{definition}

\begin{remark}
If we interpret the semigroup $\{ e^{t\cA_0} \}_{t\geq 0}$ associated to the linearization around the static N\'eel wall as an evolution system, namely,
\[
\cU_0(t,s) := e^{(t-s) \cA_0}, \qquad 0 \leq s < t
\]
(here $\cA_0$ is constant in $t$ and therefore $T$-periodic), then we can define its monodromy map as simply,
\begin{equation}
\label{defM0}
\cM_0 := \cU_0(T,0) = e^{T\cA_0}.
\end{equation}
\end{remark}

\begin{lemma}
\label{lemsimpleM0}
$\mu = 1$ is a simple isolated eigenvalue of $\cM_0$.
\end{lemma}
\begin{proof}
From Proposition \ref{propA0} (a) above (see also Lemma 5.5 in Capella \emph{et al.} \cite{CMMP24}) we know that
\[
\Theta_0 = (\partial_x \theta_0, 0) \in \ker \cA_0 \subset D_0,
\]
and that $\lambda = 0$ is a simple eigenvalue of $\cA_0$. Moreover, $\lambda = 0$ is isolated as an eigenvalue of $\cA_0$ inasmuch as
\begin{equation}
\label{sigmaA0}
\sigma(\cA_0) \subset \{ 0 \} \cup \{ z \in \C \, : \, \Re z \leq \zeta_0(\nu) < 0 \},
\end{equation}
where $\zeta_0(\nu) > 0$ depends only on the physical parameter $\nu > 0$ (see Theorem 5.1 in \cite{CMMP24}). This implies that $\mu = 1 \in \sigma(\cM_0)$ and that it is a simple isolated eigenvalue. Indeed, from standard semigroup theory it is known that $e^{t \cA_0} \Theta_0$ is a solution to
\[
\frac{d}{dt} \Big( e^{t \cA_0} \Theta_0 \Big) = \cA_0 e^{t \cA_0} \Theta_0 = e^{t \cA_0} \cA_0 \Theta_0 = 0.
\]
Thus, $e^{t \cA_0} \Theta_0 = \Theta_0$ for all $t \geq 0$. In particular, for $t = T$ one has
\[
\cM_0 \Theta_0 = e^{T \cA_0} \Theta_0 = \Theta_0,
\]
yielding $\mu = 1 \in \ptsp(\cM_0)$. That $\mu = 1$ is a simple eigenvalue follows from the simplicity of $\lambda = 0$ as an eigenvalue of $\cA_0$: any non-trivial Jordan chain for $\mu = 1 \in \ptsp(\cM_0)$ would imply the existence of a non-trivial Jordan chain for $\lambda = 0 \in \ptsp(\cA_0)$, which is impossible (recall that the generalized invariant spaces of $\cA_0$ are invariant under $e^{t \cA_0}$ and vice versa). Finally, \eqref{sigmaA0} readily implies that 
\[
\sigma(\cM_0) \subset \{ 1 \} \cup B_{r_0}(0),
\]
where $B_{r_0}(0) = \{ z \in \C \, : \, |z| < r_0 < 1\}$, and $r_0$ is any radius such that $0 <\exp(- \zeta_0(\nu)T) < r_0 < 1$. This is a consequence of ${\cA_\perp} = {\cA_0}_{|H^1 \times L^2_\perp}$, the restriction of $\cA_0$ to $\text{span} \{ \Theta_0 \}^{\perp}$ in $H^1 \times L^2$, satisfying a uniform resolvent estimate of the form 
\[
\sup_{\Re z > 0} \| (z - {\cA_\perp} )^{-1} \| < \infty
\]
(see estimate (6.15) in \cite{CMMP24}) and, upon application of the Gearhart-Pr\"uss theorem (cf. \cite{CrL03,Gea78,Pr84}), the spectral mapping theorem then holds for the operator $\cA_0$, namely, $\sigma(e^{t \cA_0}) \backslash \{0\} = e^{t \sigma(\cA_0)}$ for all $t \geq 0$ (see Engel and Nagel \cite{EN06}, Corollary 2.10, p. 183). Therefore, $\lambda \in \sigma(\cA_0)$ if and only if $\mu = e^{\lambda T} \in \sigma(\cM_0)$ with $|\mu| =  \exp((\Re \lambda) T) |\exp(i (\Im \lambda) T)| \leq \exp(-\zeta_0(\nu)T) < 1$. Hence, $\mu = 1$ is also isolated from the rest of the spectrum as an eigenvalue of $\cM_0$. This proves the lemma.
\end{proof}

\begin{corollary}
\label{corGammas}
We can choose radii $r_1 > 0$ and $0 < r_0 < 1$ such that the circles
\[
\begin{aligned}
\gamma_0 &:= \{ z \in \C \, : \, |z| = r_0 \} = \partial B_{r_0}(0), \\
\gamma_1 &:= \{ z \in \C \, : \, |z-1| = r_1 \} = \partial B_{r_1}(1),
\end{aligned}
\]
satisfy $\gamma_1, \gamma_0 \subset \rho(\cM_0)$ and 
\[
\begin{aligned}
\sigma(\cM_0) \cap B_{r_1}(1) &= \{1\}, \\
\sigma(\cM_0) \backslash \{1\} &\subset B_{r_0}(0).
\end{aligned}
\]
Moreover, the Riesz projector
\begin{equation}
\label{RieszP1M0}
\left \{
\begin{aligned}
\cP_1 &:= \frac{1}{2 \pi i} \int_{\gamma_1} (z - \cM_0)^{-1} \, dz,\\
\cP_1 &: H^1 \times L^2 \to H^2 \times H^1,
\end{aligned}
\right.
\end{equation}
has $\rank(\cP_1) = \dim \ran (\cP_1) = 1$.
\end{corollary}
\begin{proof}
From the proof of Lemma \ref{lemsimpleM0} we know that
\[
0 < \eta := \sup \big\{ |\lambda| \, : \, \lambda \in \sigma(\cM_0) \backslash \{1 \} \big\} \leq e^{ - \zeta_0(\nu) T} < 1.
\]
Therefore, we can choose any positive radius $r_0$ such that $0 < \eta < r_0 < 1$ and any other radius $r_1$ with $0 < r_1 < 1 - r_0$, so that the circle $\gamma_0$ is contained in $\rho(\cM_0)$ and $\sigma(\cM_0) \backslash \{ 1 \} \subset B_{r_0}(0) = \{ z \in \C \, : \, |z| < r_0 \}$. Moreover, the circle $\gamma_1$ contains the isolated eigenvalue $\mu = 1$ of $\cM_0$ in its interior and no other element of $\sigma(\cM_0)$, $\gamma_1 \in \rho(\cM_0)$, and $\gamma_1$ and $\gamma_0$ are completely disjoint; see Figure \ref{fig2} below. Hence, we can define the Riesz operator in \eqref{RieszP1M0} and
\[
\mathrm{Ran}(\cP_1) = \bigcup_{k \geq 1} \ker \big( (\cM_0 - I )^k \big),
\]
has finite dimension, equal to the algebraic multiplicity of $\mu = 1$ as an eigenvalue of $\cM_0$. Since $\mu = 1$ is simple, we obtain $\mathrm{rank}(\cP_1) = 1$.
\end{proof}

\begin{remark}
From the proof of Corollary \ref{corGammas} above, we also have that $\gamma_0 \in \rho(\cM_0)$ and, likewise, we can define the Riesz projector
\begin{equation}
\label{RieszP0M0}
\left \{
\begin{aligned}
\cP_0 &:= \frac{1}{2 \pi i} \int_{\gamma_0} (z - \cM_0)^{-1} \, dz,\\
\cP_0 &: H^1 \times L^2 \to H^2 \times H^1.
\end{aligned}
\right.
\end{equation}
Note that $\gamma_0$ contains the rest of the spectrum of $\cM_0$ in its interior.
\end{remark}

\subsection{Proof of spectral stability}

We now prove the persistence of the spectral partition of the monodromy map for $\vep$ sufficiently small, yielding spectral stability. The following result is the main part of the proof.

\begin{lemma}
There exists $0 < \vep_1 < \vep_0$, sufficiently small, such that $\mu = 1 \in \sigma(\cM_{\vep})$ with algebraic multiplicity equal to one, for all $\vep \in (-\vep_1, \vep_1)$.
\end{lemma}
\begin{proof}
Let $\vep \in (-\vep_0, \vep_0)$ and consider the propagator $\cU$ from Theorem \ref{thmexistpropag}. By elementary properties of the propagator (see \hyperref[E3]{\rm{(E$_3$)}} in Theorem \ref{thmexistpropag}) and of the generator $\cA_0$ we have
\[
\begin{aligned}
\frac{\partial}{\partial s} \big( \cU(t,s) e^{s \cA_0}\big) &= - \cU(t,s) \cA(s) e^{s \cA_0} + \cU(t,s) \cA_0 e^{s \cA_0}\\
&= \cU(t,s) \big( \cA_0 - \cA(s) \big) e^{s \cA_0},
\end{aligned}
\]
for all $0 \leq s \leq t \leq T$. Evaluating in $t = T$ and integrating in $s \in (0,T)$ we arrive at
\[
\begin{aligned}
\cU(T,T) e^{T \cA_0} - \cU(T,0) &= - \int_0^T \cU(T,s) \big( \cA_0 - \cA(s) \big) e^{s \cA_0} \, ds \\
&= - \int_0^T \cU(T,s) \cB_\vep (s) e^{-s \cA_0} \, ds,
\end{aligned}
\]
where the family $\{ \cB_\vep(t) \}_{t \in [0,T]}$ is defined in \eqref{defcB}. Since $\cU(T,T) = I$, as a result we have found the expression for the difference of the monodromy operators,
\begin{equation}
\label{difMs}
\cM_\vep - \cM_0 = \int_0^T \cU(T,s) \cB_\vep(s) e^{s \cA_0} \, ds.
\end{equation}
Upon estimation of its operator norm we find that
\[
\| \cM_\vep - \cM_0 \| \leq \int_0^T \| \cU(T,s) \| \| \cB_\vep(s) \| \| e^{s \cA_0} \| \, ds \leq C | \vep| e^{\omega T} \int_0^T e^{(\omega + C|\vep|)(s-T)} \, ds,
\]
in view of \hyperref[E1]{\rm{(E$_1$)}}, the uniform bound $\| \cB_\vep (s) \| \leq C |\vep|$ and the quasicontractivity of the semigroup $e^{t \cA_0}$. We claim that
\[
\int_0^T e^{(\omega + C|\vep|)(s-T)} \, ds \leq C_T + O((|\omega|+|\vep|)T^2),
\]
for some positive constant $C_T = O(T) > 0$. Indeed, in the case when $\omega \geq - C|\vep|$ we clearly have
\[
\int_0^T e^{(\omega + C|\vep|)(s-T)} \, ds \leq T.
\]
In the case when $\varrho := - (\omega + C|\vep|) > 0$ there holds
\[
\int_0^T e^{(\omega + C|\vep|)(s-T)} \, ds = (\omega + C|\vep|)^{-1} \big( 1- e^{-(\omega + C|\vep|)T}\big) = \varrho^{-1}(e^{\varrho T}-1) = T +  O(\varrho T^2).
\]
Hence we arrive at the estimate,
\begin{equation}
\label{boundMM0}
\| \cM_\vep - \cM_0 \| \leq \widetilde{C}|\vep| + O(|\vep|^2 ),
\end{equation}
for some constant $\widetilde{C} = \widetilde{C}(T,\omega) > 0$. Let us now define the Riesz projector,
\begin{equation}
\label{defP1vep}
\left \{
\begin{aligned}
\cP_1(\vep) &:= \frac{1}{2 \pi i} \int_{\gamma_1} (z - \cM_\vep)^{-1} \, dz,\\
\cP_1(\vep) &: H^1 \times L^2 \to H^2 \times H^1,
\end{aligned}
\right.
\end{equation}
where the integration is performed over the same circle $\gamma_1$ as in \eqref{RieszP1M0}. We claim that for $0< \vep \ll 1$ sufficiently small, $\gamma_1 \subset \rho(\cM_\vep)$ and $\cP_1(\vep)$ is analytic in $\vep$.

For that purpose we use the Neumann expansion
\[
\begin{aligned}
(z - \cM_\vep)^{-1} &= \big(z - \cM_0 + (\cM_0 - \cM_\vep) \big)^{-1} \\
&= \Big( \sum_{n=0}^\infty (\cM_0 - \cM_\vep)^n (z - \cM_0 )^{-n} \Big) (z - \cM_0)^{-1} ,
\end{aligned}
\]
for $z \in \gamma_1$. The series converges uniformly for $\| \cM_0 - \cM_\vep \| \| (z- \cM_0)^{-1}\| < 1$. Since $\| (z- \cM_0)^{-1}\|$ is bounded for $z \in \gamma_1$, then there exists $0 < \vep_1 < \vep_0$ sufficiently small such that if $|\vep| < \vep_1$ then 
\[
\| \cM_0 - \cM_\vep \| \| (z- \cM_0)^{-1}\| \leq C(\widetilde{C}|\vep| + O(|\vep|^2 )) < 1.
\]
Thus, the series converges uniformly, $\cP_1(\vep)$ is analytic in $\vep$ and $\gamma_1 \subset \rho(\cM_\vep)$. Moreover, $\cP_1(\vep)$ is a Riesz operator such that
\[
\| \cP_1(\vep) - \cP_1 \| \leq C \sup_{z \in \gamma_1} \| (z - \cM_\vep)^{-1} - (z - \cM_0)^{-1} \| \to 0,
\]
as $\vep \to 0$ by continuity of the resolvent. Hence, for $|\vep|$ sufficiently small, $\mathrm{rank}(\cP_1(\vep)) = \mathrm{rank}(\cP_1)$. 

We conclude that for $|\vep|$ small enough, there exists only one simple eigenvalue $\mu(\vep)$ of $\cM_\vep$ inside $\gamma_1$. By analyticity of $\cP_1(\vep)$, this eigenvalue is also analytic in $\vep$ (see Theorem VII.10.8 in Kato \cite{Kat80}) and it is associated to a certain eigenfunction, say $\overline{\Theta}(\vep) \in \text{Ran}(\cP_1(\vep)) \subset H^2 \times H^1$. However, from Corollary \ref{corpropsol} we already know that $1 \in \sigma(\cU(T,0)) = \sigma(\cM_\vep)$ with eigenfunction $\overline{\Theta} (\cdot,0) \in H^2 \times H^1$. This implies that $\mu(\vep) \equiv 1$ for all $0 < |\vep| < \vep_1$ sufficiently small and that this eigenvalue is simple.  
\end{proof}

\begin{remark}
By the same arguments,
\[
\cP_0(\vep) := \frac{1}{2 \pi i} \int_{\gamma_0} (z - \cM_\vep)^{-1} \, dz,
\]
is analytic for $|\vep|$ sufficiently small and $\text{rank}(\cP_0(\vep)) = \text{rank}(\cP_0)$, so that $\gamma_0$ encloses $\sigma(\cM_\vep) \backslash \{ 1 \}$ (see Figure \ref{fig2}).
\end{remark}

\begin{figure}[t]
\begin{center}
\includegraphics[scale=0.35]{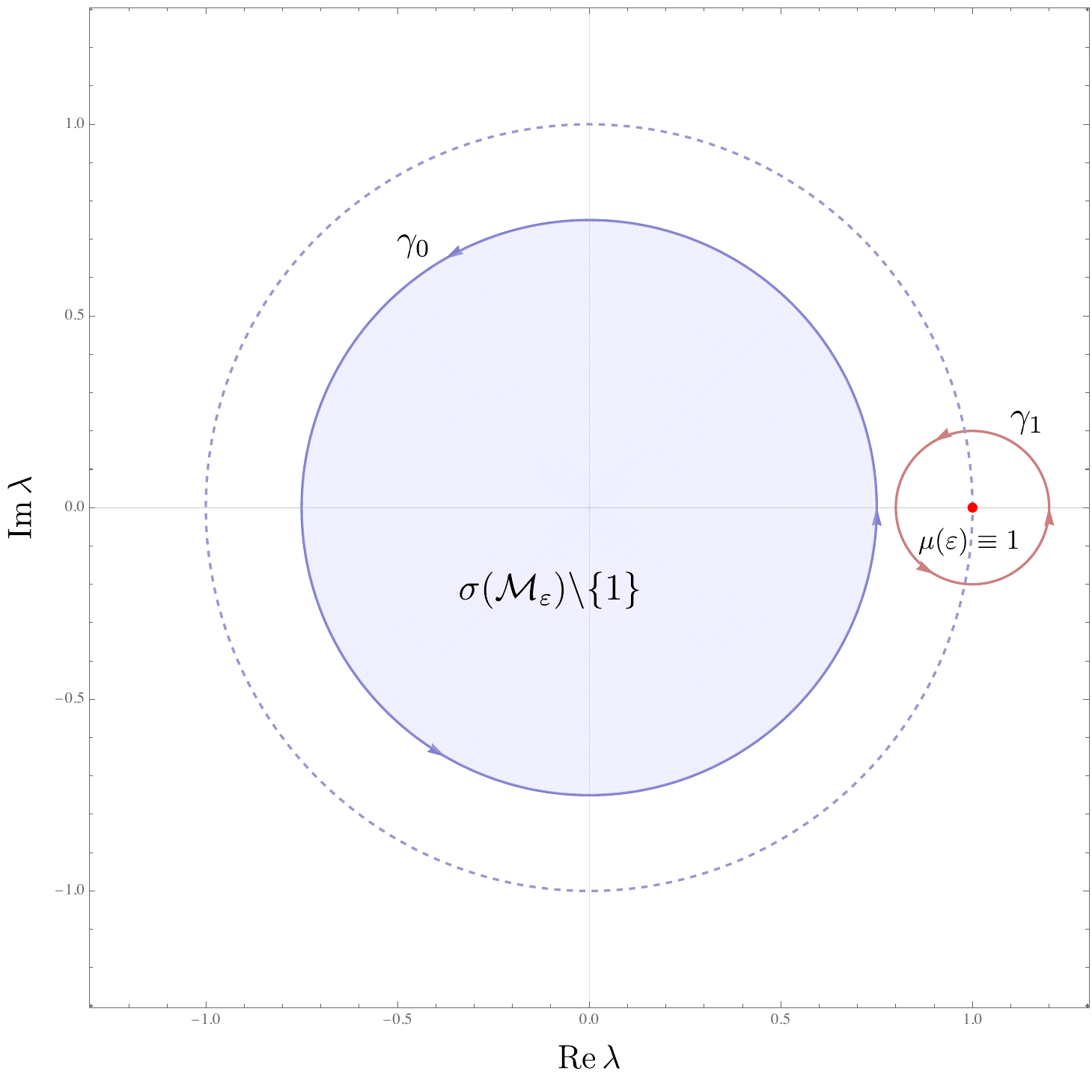}
\caption{\small{Contours $\gamma_0$ (blue) and $\gamma_1$ (red) in the complex plane, encircling $\sigma(\cM_0)\backslash \{ 1 \}$ and $\{1\}$, respectively. The unit circle is represented by a dotted line.  By taking $0 < \vep \ll 1$ sufficiently small, the interiors of these contours also contain $\sigma(\cM_\vep)\backslash \{ 1 \}$ and $\mu(\vep) = 1$, respectively. It is to be observed that $\mu(\vep) = 1$ persists as an eigenvalue of $\cM_\vep$ for $\vep$ sufficiently small (color online). }
}
\label{fig2}
\end{center}
\end{figure}

Consequently, we have the following result.
\begin{corollary}
There exists $0 < \vep_1 < \vep_0$ sufficiently small such that
\[
\begin{aligned}
\sigma(\cM_\vep) \cap B_{r_1}(1) &= \{1\}, \\
\sigma(\cM_\vep) \backslash \{1\} &\subset B_{r_0}(0),
\end{aligned}
\]
for all $0 < |\vep| < \vep_1$.
\end{corollary}

These observations yield the spectral stability property.
\begin{theorem}[spectral stability of the time-periodic N\'eel wall]
There exists $\vep_1 > 0$ sufficiently small such that for $0 < |\vep| < \vep_1$ there holds
\[
\Sigma_F = \{ \lambda \in \C \, : \, e^{T\lambda} \in \sigma(\cM_\vep) \} \subset \{ \lambda \in \C \, : \, \Re \lambda \leq 0 \}.
\]
\end{theorem}
\begin{proof}
From the previous Corollary, $\sigma(\cM_\vep) \backslash \{1\} \subset B_{r_0}(0)$ with $0 < r_0 < 1$. Therefore, $e^{T\lambda} \in \sigma(\cM_\vep)$ implies that either $e^{T\lambda} = 1$ or $e^{T\lambda} \in B_{r_0}(0)$ with $r_0 < 1$, yielding $\Re \lambda \leq 0$, as claimed.
\end{proof}
%
%
%
%

%
%
%
%
%
%
%
%
%
%
%
%

\section*{Acknowledgements}

L. Morales and R. G. Plaza acknowledge the hospitality of the Department of Mathematics at RWTH Aachen University during research visits in September and October 2025, respectively, when this work was partially carried out.


\section*{Summary Statement}
\subsection*{Funding declarations}
V. Linse and C. Melcher  were supported by the Deutsche Forschungsgemeinschaft (DFG, German Research Foundation) - Project number 442047500 through the Collaborative Research Center “Sparsity and Singular Structures” (SFB 1481). The work of A. Capella and R. G. Plaza was fully supported by SECIHTI, M\'exico, grant CF-2023-G-122. The work of L. Morales was supported by SECIHTI, M\'exico, through the Program ``Estancias Posdoctorales por M\'exico 2022''. 


\subsection*{Conflict of interest} The authors declare no conflict of interest.
\subsection*{Author contributions} The authors contributed to and reviewed all article sections equally.
\subsection*{Ethics declaration} Not applicable.
\subsection*{Data Availability} No datasets were generated or analyzed during the current study.




%

\def\cprime{$'\!\!$}

\end{document}